\documentclass[11pt]{article}

\usepackage{amssymb}
\usepackage{amsmath}
\usepackage{a4wide}
\usepackage{mathptmx}
\usepackage{amsthm}
\usepackage{hyperref}
\theoremstyle{plain}
\newtheorem{thm}{Theorem}[section] % reset theorem numbering for each chapter
\newtheorem{defn}[thm]{Definition} % definition numbers are dependent on theorem numbers
\newtheorem{lem}[thm]{Lemma} % same for lemma
\newtheorem{pro}[thm]{Proposition}

\theoremstyle{definition}

\newtheorem{rem}[thm]{Remark}

\begin{document}
	\begin{center}
		\section*{On quasi-infinitely divisible random measures}
		\subsection*{Riccardo Passeggeri\footnote{\noindent LPSM, Sorbonne University. Email: riccardo.passeggeri@gmail.com}}
		\today
	\end{center}

 			\begin{abstract}
 			Quasi-infinitely divisible (QID) distributions have been recently introduced by Lindner, Pan and Sato (\textit{Trans.~Amer.~Math.~Soc.}~\textbf{370}, 8483-8520 (2018)). A random variable $X$ is QID if and only if there exist two infinitely divisible (ID) random variables $Y$ and $Z$ s.t.~$X+Y\stackrel{d}{=}Z$ and $Y$ is independent of $X$. In this work, we show that a family of QID completely random measures (CRMs) is dense in the space of all CRMs with respect to convergence in distribution. We further demonstrate that the elements of this family posses a L\'{e}vy-Khintchine formulation and that there exists a one to one correspondence between their law and certain characteristic pairs. We prove the same results also for the class of point processes with independent increments. In the second part of the paper, we show the relevance of these results in the general Bayesian nonparametric framework based on CRMs developed by Broderick, Wilson and Jordan (\textit{Bernoulli}, \textbf{24}, 3181-3221 (2018)). \\\\
 				\textbf{Keywords:} quasi-infinite divisibility, completely random measure, dense class, nonparametric Bayesian analysis, automatic conjugacy.\\\\
 			 \textbf{MSC (2010):} 60E07, 60G57, 60A10, 62F15
 			\end{abstract}

\tableofcontents

\section{Introduction}
A random measure $\xi$ on $S$, with underlying probability space $(\Omega,\mathcal{F},\mathbb{P})$, is a function $\Omega\times\textbf{S}\rightarrow[0,\infty]$, such that $\xi(\omega,B)$ is a $\mathcal{F}$-measurable in $\omega\in\Omega$ for fixed $B$ and a locally finite measure in $B\in\textbf{S}$ for fixed $\omega$. Completely random measures (CRMs) have the additional property that for any disjoint $B_{1}, B_{2}, . . . , B_{k}\in \textbf{S}$, $k\in\mathbb{N}$, the random variables $\xi(B_{1}),\xi(B_{2}),...,\xi(B_{k})$ are
independent. CRMs, also called independently scattered random measures or random measures with independent increments, have a fundamental role in nonparametric Bayesian analysis; as Ghosal and van der Vaart affirm in their recent book (see \cite{Vaart}) CRMs ``arise as priors, or building blocks for priors, in \textit{many} Bayesian nonparametric applications".

Completely random measures have a long history which is inextricably linked with the one of infinitely divisible distributions. In 1967, Kingman \cite{King} proved a very appealing and useful representation theorem for all CRMs. He showed that any CRM $\xi$ is almost surely given by the sum of three components: one deterministic,  one  concentrated  on  a  fixed  set   of  atoms,  and  one  concentrated   on   a   random  set   of   atoms. He further showed that the last component, which he called the ordinary component, is fully determined by a Poisson point process: $\xi_{ord}(B)=\int_{(0,\infty)}x\eta(B\times dx)$, where $\eta$ is a Poisson point process on $S\times(0,\infty)$. The Poisson point process is the prime example of infinitely divisible CRM.

Infinitely divisible (ID) distributions have an even longer history that goes back to the work of L\'{e}vy, Kolmogorov and De Finetti, among others. They constitute one of the most studied classes of probability distributions. One of their most attractive properties is that their characteristic function have an explicit formulation, called the L\'{e}vy-Khintchine formulation, written in terms of three mathematical objects. These are the drift, which is a real valued constant, the Gaussian component, which is a non-negative constant, and the L\'{e}vy measure, which is a measure on $\mathbb{R}$ satisfying an integrability condition and with no mass at $\{0\}$. Gaussian and Poisson distributions are examples of this class.

In 2018, in \cite{LPS} Sato, Lindner and Pan introduced the class of quasi-infinitely divisible (QID) distributions. A QID random variable is defined as follows: a random variable $X$ is QID (namely has a QID distribution) if and only if there exist two ID random variables $Y$ and $Z$ s.t.~$X+Y\stackrel{d}{=}Z$ and $Y$ is independent of $X$. QID distributions are like ID distributions except for the fact that the L\'{e}vy measure is now allowed to take negative values. In other words, a QID distribution has a L\'{e}vy-Khintchine formulation which is uniquely determined by a drift, a Gaussian component and by a `signed measure' (more precisely a real valued set function) called the quasi-L\'{e}vy measure. Any ID distribution is QID, but the converse is not always true.

In \cite{LPS}, the authors show that QID distributions are dense in the space of all probability distributions with respect to weak convergence and that distributions concentrated on the integers (or any shift and dilation of them) are QID if and only if their characteristic functions have no zeros, among other results. Further theoretical results have been achieved in \cite{Berger,Khartov,Pass,Pass-bim}. In \cite{Pass}, the QID framework is extended to real-valued random noises and stochastic processes. QID distributions have already shown to have an impact in various fields: from mathematical physics, see \cite{Physics2} and \cite{Physics1}, to number theory, see \cite{Naka} and \cite{Naka2}, and to insurance mathematics, see \cite{Zhang}.

The first main contribution of this paper is the density result for QID random measures. We prove that a certain class of QID completely random measures (CRMs), which we denote by $\mathcal{A}$, is dense with respect to convergence in distribution (precisely in both weak and vague convergence) in the space of all CRMs, also know as random measures with independent increments or as independently scattered random measures. This result extends the density result in \cite{LPS} to the infinite dimensional setting of CRMs. 

The class $\mathcal{A}$ have quite remarkable features. First, as any CRM they have an almost sure representation in terms of an `atomless' ID component and an `atomic' one. Second, the number of atoms is finite. Third, these random measures are almost surely finite and even more their atomless component has finite L\'{e}vy measure.

Moreover, for the elements of this class, we are able to show an explicit spectral representation, namely the L\'{e}vy-Khintchine formulation, and prove that there exists a unique one-to-one correspondence between them and pairs of deterministic measures satisfying certain conditions, which we call characteristic pairs. We prove all these results also for the class of point processes with independent increments, of which the Poisson point process is an example.

With these results this paper shows that the fixed component of a CRM, which has been left out in Kingman's analysis and in the theory of CRM in general, have exactly the same nice representation results as the widely studied ordinary component. Thus, not only there is no real need of leaving out of the analysis the fixed component (as Kingman graphically says, fixed atoms can be removed by simple surgery), but this might also be dangerous since in many applications the fixed component has an irreplaceable role. This will also appear evident in the Bayesian setting we discuss in this work (see also \cite{Bro1}).

In the last section we investigate the impact of these results in the nonparametric Bayesian statistical framework presented by Broderick, Wilson and Jordan in \cite{Bro1} based on CRMs (see also \cite{Bro2}). In particular, we consider priors to be given by elements in $\mathcal{A}$ (with quasi-L\'{e}vy measure having a particular structure). We show that they are dense in the space of priors considered in \cite{Bro1} and \cite{Bro2} with respect to convergence in distribution, thus showing also that our density result is flexible enough to adjust to various assumptions/settings. Second, we present explicit formulations for their posterior distributions. Third, when focusing on point processes, we prove automatic conjugacy for all the elements of $\mathcal{A}$ under the only condition that the characteristic function of the posterior distribution has no zeros. This condition is satisfied in many situations and the result is more general than the the one of \cite{Bro1} which is based on the exponential structure of the likelihood.

We remark that the general nature of our results allow them to be applied in many Bayesian settings. Thus, the choice of the work of Broderick, Wilson and Jordan \cite{Bro1} represents a first easy example.

The paper is structured as follows. Section \ref{Sec-Notation} concerns with the notations and some preliminaries. In Section \ref{Sec-Atomless} we provide the density results for CRMs and in Subsection \ref{Subsec-Density-Point} the one for point processes with independent increments. In Section \ref{Sec-Properties}, we show various properties for the classes of QID random measures and QID point processes presented in Section \ref{Sec-Atomless}. In particular we present the L\'{e}vy-Khintchine formulation and the one-to-one correspondence of these random measures with their unique characteristic pair. In Section \ref{Sec-Bay}, we present the Bayesian setting and the relative results: computation of the posterior, convergence results for the posterior, and automatic conjugacy.
\section{Notation and Preliminaries}\label{Sec-Notation}
By a measure on a measurable space $(X,\mathcal{G})$ we always mean a positive measure on $(X,\mathcal{G})$, \textit{i.e.}~a $[0,\infty]$-valued $\sigma$-additive set function on $\mathcal{G}$ that assigns the value $0$ to the empty set. For a non-empty set $X$, by $\mathcal{B}(X)$ we mean the Borel $\sigma$-algebra of $X$, unless stated differently. The law and the characteristic function of a random variable $X$ will be denoted by $\mathcal{L}(X)$ and by $\hat{\mathcal{L}}(X)$, respectively. For two measurable spaces $(X,\mathcal{G})$ and $(Y,\mathcal{F})$, we denote by $\mathcal{G}\otimes\mathcal{F}$ the product $\sigma$-algebra of $\mathcal{G}$ and $\mathcal{F}$, and by $\mathcal{G}\times\mathcal{F}$ their Cartesian product. Let us recall some definitions.

\begin{defn}[extended signed measure]\label{Def-signedmeasure}
	Given a measurable space $(X, \Sigma)$, that is, a set $X$ with a $\sigma$-algebra $\Sigma$ on it, an extended signed measure is a function $\ \mu :\Sigma \to {\mathbb {R}}\cup \{\infty ,-\infty \}$ s.t.~$\mu (\emptyset )=0$ and $\mu$ is $\sigma$-additive, that is, it satisfies the equality $ \mu \left(\bigcup _{{n=1}}^{\infty }A_{n}\right)=\sum _{{n=1}}^{\infty }\mu (A_{n})$ where the series on the right must converge in ${\mathbb {R}}\cup \{\infty ,-\infty \}$ absolutely (namely the value of the series is independent of the order of its elements), for any sequence $A_{1}, A_{2},...$ of disjoint sets in $\Sigma$.
\end{defn}
\noindent As a consequence any extended signed measure can take plus or minus infinity as value, but not both. In this work, we use the term `signed measure' for an extended signed measure. Further, the \textit{total variation} of a signed measure $\mu$ is defined as the measure $|\mu|:\Sigma\rightarrow [0, \infty]$ defined by
\begin{equation}\label{def-totalvariation}
|\mu|(A):=\sup\sum_{j=1}^{\infty}|\mu(A_{j})|,
\end{equation}
where the supremum is taken over all the partitions $\{A_{j}\}$ of $A\in\Sigma$. The total variation $|\mu|$ is finite if and only if $\mu$ is finite. Let us recall the definition of a signed bimeasure.
\begin{defn}[Signed bimeasure]
	Let $(X,\Sigma)$ and $(Y,\Gamma)$ be two measurable spaces. A \textnormal{signed bimeasure} is a function $M:\Sigma\times\Gamma\rightarrow[-\infty,\infty]$ such that:
	\\\textnormal{(i)} the function $A\rightarrow M(A,B)$ is a signed measure on $\Sigma$ for every $B\in\Gamma$,
	\\\textnormal{(i)} the function $B\rightarrow M(A,B)$ is a signed measure on $\Gamma$ for every $A\in\Sigma$.
	
\end{defn}

Let $S$ be a separable and complete metric space with Borel $\sigma$-algebra $\textbf{S}$ and let $\hat{\mathbf{S}}$ be the ring composed by bounded Borel sets in $S$. The triplet $(S,\textbf{S},\hat{\textbf{S}})$ is called localised Borel space (see page 19 in \cite{Kallenberg2}).
\begin{defn}[random measure]\label{def-rm}
	A random measure $\xi$ on $S$, with underlying probability space $(\Omega,\mathcal{F},\mathbb{P})$, is a function $\Omega\times\textbf{S}\rightarrow[0,\infty]$, such that $\xi(\omega,B)$ is a $\mathcal{F}$-measurable in $\omega\in\Omega$ for fixed $B$ and a locally finite measure in $B\in\textbf{S}$ for fixed $\omega$.
\end{defn}
\begin{defn}
	[completely random measure] A completely random measure (CRM) $\xi$ is a random measure s.t.~for any disjoint $B_{1}, B_{2}, . . . , B_{k}\in \textbf{S}$, $k\in\mathbb{N}$, the random variables $\xi(B_{1}),\xi(B_{2}),...,\xi(B_{k})$ are
	independent. CRMs are also called independently scattered random measure or random measure with independent increments.
\end{defn}
\begin{defn}[diffuse random measure]
	Using the notation of the previous definition, we say that a random measure $\xi$ on $S$ is \textnormal{diffuse} if $\xi(\omega,B)$ is a locally finite diffuse measure in $B\in\textbf{S}$ for fixed $\omega$.
\end{defn}
\begin{rem}
	\textnormal{Term \textit{finite} for random measures stands for \textit{a.s.~finite}. Thus, for a \textit{finite random measure}~we mean an \textit{a.s.~finite random measure}.} 
\end{rem}

For a random measure $\xi$ on a Polish space $X$, $x\in X$ is a fixed atom of $\xi$ if and only if $\mathbb{P}(|\xi(\{x\})|>0)>0$. Further, a random measure $\xi$ is called atomless if $\xi(\{x\})\stackrel{a.s.}{=}0$ for every $x\in X$. The atomless condition is for random measures what the continuity in probability is for continuous time stochastic processes. We remark that an atomless random measure is not necessarily a diffuse random measure (see Corollary 12.11 in \cite{Kallenberg}). For example, think of a Poisson point process with $\mathbb{E}[\xi(s)]\equiv0$, like the homogeneous Poisson point process, which has no fixed atoms but it is not diffuse.

Now, we introduce the concept of a quasi-L\'{e}vy type measure. We start with the following definition, which we recall from \cite{LPS}:
\begin{defn}\label{def1}
	Let $\mathcal{B}_{r}(\mathbb{R}):=\{B\in\mathcal{B}(\mathbb{R})| B \cap(-r, r) = \emptyset\}$ for $r > 0$ and $\mathcal{B}_{0}(\mathbb{R}):= \bigcup_{r>0} \mathcal{B}_{r}(\mathbb{R})$ be the class of all
	Borel sets that are bounded away from zero. Let $\nu : \mathcal{B}_{0}(\mathbb{R})\rightarrow\mathbb{R}$ be a function such that
	$\nu_{|\mathcal{B}_{r}(\mathbb{R})}$ is a finite signed measure for each $r > 0$ and denote the total variation, positive and negative part of $\nu_{|\mathcal{B}_{r}(\mathbb{R})}$ by $|\nu_{|\mathcal{B}_{r}(\mathbb{R})}|$, $\nu^{+}_{|\mathcal{B}_{r}(\mathbb{R})}$ and $\nu^{-}_{|\mathcal{B}_{r}(\mathbb{R})}$ respectively. Then the \textnormal{total variation} $|\nu|$, the \textnormal{positive part} $\nu^{+}$ and the \textnormal{negative part} $\nu^{-}$ of $\nu$ are defined to be the unique measures on $(\mathbb{R},\mathcal{B}(\mathbb{R}))$ satisfying
	\begin{equation*}
	|\nu|(\{0\})=\nu^{+}(\{0\})=\nu^{-}(\{0\})=0
	\end{equation*}
	\begin{equation*}
	\text{and}\quad|\nu|(A)=|\nu_{|\mathcal{B}_{r}(\mathbb{R})}|,\,\,\nu^{+}(A)=\nu_{|\mathcal{B}_{r}(\mathbb{R})}^{+}(A),\,\,\nu^{-}(A)=\nu_{|\mathcal{B}_{r}(\mathbb{R})}^{-}(A),
	\end{equation*}
	for $A\in\mathcal{B}_{r}(\mathbb{R})$, for some $r>0$.
\end{defn}
As mentioned in \cite{LPS}, $\nu$ is not a a signed measure because it is defined on $\mathcal{B}_{0}(\mathbb{R})$, which is not a $\sigma$-algebra. In the case it is possible to extend the definition of $\nu$ to $\mathcal{B}(\mathbb{R})$ such that $\nu$ is a signed measure then we will identify $\nu$ with its extension to $\mathcal{B}(\mathbb{R})$ and speak of $\nu$ as a signed measure. Moreover, the uniqueness of $|\nu|$, $\nu^{+}$ and $\nu^{-}$ is ensured by an application of the Carath\'{e}odory's extension theorem (see Lemma 2.14 in \cite{Pass}). Further, notice that $\mathcal{B}_{0}(\mathbb{R})= \{B\in\mathcal{B}(\mathbb{R}):0\notin \overline{B} \}\neq \{B\in\mathcal{B}(\mathbb{R}):0\notin B \}$ (see Remark 2.6 in \cite{Pass}).
\begin{defn}[quasi-L\'{e}vy type measure, quasi-L\'{e}vy measure, QID distribution, from \cite{LPS}]
	A \textnormal{quasi-L\'{e}vy type measure} is a function $\nu: \mathcal{B}_{0}(\mathbb{R})\rightarrow \mathbb{R}$ satisfying the
	condition in Definition \ref{def1} and such that its total variation $|\nu|$ satisfies $\int_{\mathbb{R}} (1\wedge x^{2} ) |\nu|(dx) <\infty$. Let $\mu$ be a probability distribution on $\mathbb{R}$. We say that $\mu$ is \textnormal{quasi-infinitely divisible} if its characteristic function has a representation
	\begin{equation*}
	\hat{\mu}(\theta)=\exp\left(i\theta \gamma-\frac{\theta^{2}}{2}a+\int_{\mathbb{R}}\left(e^{i\theta x}-1-i\theta\tau(x)\right)\nu(dx)\right)
	\end{equation*}
	where $a, \gamma \in \mathbb{R}$ and $\nu$ is a quasi-L\'{e}vy type measure. The characteristic triplet $(\gamma,a,\nu)$
	of $\mu$ is unique, and $a$ and $\gamma$ are called the \textnormal{Gaussian variance} and the \textnormal{drift} of $\mu$, respectively. A quasi-L\'{e}vy type measure $\nu$ is called \textnormal{quasi-L\'{e}vy measure}, if additionally there exist a quasi-infinitely divisible distribution $\mu$ and some $a,\gamma\in\mathbb{R}$	such that $(\gamma,a,\nu)$ is the characteristic triplet of $\mu$. We call $\nu$ the quasi-L\'{e}vy measure of $\mu$.
\end{defn}
The above definition extend to the $\mathbb{R}^{d}$ case (for $d>1$) as shown in Remark 2.4 in \cite{LPS}. As pointed out in Example 2.9 of \cite{LPS}, a quasi-L\'{e}vy measure is always a quasi-L\'{e}vy type measure, while the converse is not true. Moreover, we say that a function $f$ is \textit{integrable with respect to quasi-L\'{e}vy type measure} $\nu$ if it is integrable with respect to $|\nu|$. Then, we define:
\begin{equation*}
\int_{B}fd\nu:=\int_{B}fd\nu^{+}-\int_{B}fd\nu^{-},\quad B\in\mathcal{B}(\mathbb{R}).
\end{equation*}
In this work we always keep the same order for the elements in the characteristic triplet: the first element is the drift, the second one is the Gaussian variance, and the third one is the (quasi) L\'{e}vy measure.
\begin{defn}
	[QID random measure] Let $\Lambda$ be a random measure. If $\Lambda(A)$ is a QID random variable, for every $A\in\textbf{S}$, then we call $\Lambda$ a QID random measure.
\end{defn}
We conclude with the following result on QID distributions.
\begin{thm}[Theorem 4.3.4 in \cite{Cuppens}]\label{Cupp} Let $d\in\mathbb{N}$. The characteristic triplet $(\gamma,0,\nu)$, where $\nu$ is a finite quasi-L\'{e}vy type measure, is the characteristic triplet of a QID distribution on $\mathbb{R}^{d}$ if and only if $\exp(\nu):=\sum_{n=1}^{\infty}\frac{\nu^{*n}}{n!}$ is a measure. In that case, $\mu\sim(\gamma,0,\nu)$ is given by
	\begin{equation*}
	\mu=\frac{\delta_{\gamma}*\exp(\nu)}{\exp(\nu(\mathbb{R}^{d}))}.
	\end{equation*}	
\end{thm}
\section{The density result for QID CRMs}\label{Sec-Atomless}
In this section we present the density results for QID CRMs in the space of all CRMs with respect to convergence in distribution. Let us start with some preliminaries. Let $S$ be a separable and complete metric space with Borel $\sigma$-algebra $\textbf{S}$ and let $\hat{\mathbf{S}}$ be the ring composed by bounded Borel sets in $S$. Let $\hat{C}_{S}$ be the space of all bounded continuous functions $f:S\rightarrow\mathbb{R}_{+}$ with bounded support. Let $\mathcal{M}_{S}$ be the space of locally finite measures, namely $\mu\in\mathcal{M}_{S}$ if $\mu(B)<\infty$ for every $B\in\hat{\mathbf{S}}$. The space $\mathcal{M}_{S}$ might be endowed with the vague topology, denoted by $\mathbf{B}_{\mathcal{M}_{S}}$, generated by the integration maps $\pi_{f}:\mu\mapsto\int f(x) \mu(dx)$, for all $f\in \hat{C}_{S}$. The vague topology is the coarsest topology making all $\pi_{f}$ continuous. The measurable space $(\mathcal{M}_{s},\mathbf{B}_{\mathcal{M}_{S}})$ is a Polish space. The associated notion of vague convergence denoted by $\mu_{n}\stackrel{v}{\rightarrow}\mu$ is defined by the condition $\int f(x) \mu_{n}(dx)\rightarrow\int f(x) \mu(dx)$ for all $f\in\hat{C}_{S}$.

An equivalent definition of random measure (see Definition \ref{def-rm}) is the following: a random measure $\xi$ is a measurable mapping from $(\Omega,\mathcal{F},\mathbb{P})$ to $(\mathcal{M}_{S},\mathcal{B}_{\mathcal{M}_{S}})$, where $\mathcal{B}_{\mathcal{M}_{S}}$ is the topology generated by all projection maps $\pi_{B}:\mu\mapsto\mu(B)$ with $B\in\mathbf{S}$, or, equivalently, by all integration maps $\pi_{f}$ with measurable $f\geq0$. From Lemma 4.1 in \cite{Kallenberg0} or Theorem 4.2 in \cite{Kallenberg2}, we know that $\mathcal{B}_{\mathcal{M}_{S}}$ and $\mathbf{B}_{\mathcal{M}_{S}}$ coincide. Hence it is equivalent to consider a random measure as a measurable mapping from $(\Omega,\mathcal{F},\mathbb{P})$ to $(\mathcal{M}_{S},\textbf{B}_{\mathcal{M}_{S}})$ or to $(\mathcal{M}_{S},\mathcal{B}_{\mathcal{M}_{S}})$.

The convergence in distribution of $\xi_{n}$ to $\xi$ means that $\mathbb{E}[g(\xi_{n})]\rightarrow\mathbb{E}[g(\xi)]$ for every bounded continuous function $g$ on $\mathcal{M}_{S}$, or equivalently that $\mathcal{L}(\xi_{n})\stackrel{w}{\rightarrow}\mathcal{L}(\xi)$, where for any bounded measures $\mu_{n}$ and $\mu$, the weak convergence $\mu_{n}\stackrel{w}{\rightarrow}\mu$ stands for $\int g(y)\mu_{n}(dy)\rightarrow\int g(y)\mu(dy)$ for all $g$ as above. We write $\xi_{n}\stackrel{vd}{\rightarrow}\xi$ to stress that the convergence of distribution is for random measures considered as random elements in the space $\mathcal{M}_{S}$ with vague topology. As mentioned in the previous section, in this setting an atom of a random measure $\xi$ is an element $s\in S$ such that $\mathbb{P}(\xi(\{s\})>0)>0$.

We recall now a fundamental result by Harris, see \cite{Harris}.
\begin{thm}[see Theorem 4.11 in \cite{Kallenberg2}]\label{density-T1}
	Let $\xi,\xi_{1},\xi_{2},...$ be random measures on $S$. Then these conditions are equivalent:
	\\\textnormal{(i)} $\xi_{n}\stackrel{vd}{\rightarrow}\xi$,
	\\\textnormal{(ii)} $\int f(x)\xi_{n}(dx)\stackrel{d}{\rightarrow}\int f(x)\xi(dx)$ for all $f\in\hat{C}_{S}$,
	\\\textnormal{(iiI)} $\mathbb{E}[\exp(-\int f(x)\xi_{n}(dx))]\rightarrow\mathbb{E}[\exp(-\int f(x)\xi(dx))]$ for all $f\in\hat{C}_{S}$ with $f\leq1$.	
\end{thm}
The following density result extends Theorem 4.1 in \cite{LPS}.
\begin{thm}\label{pro8}
Let $A$ be a connected interval of the real line. The class of QID distributions with finite quasi-L\'{e}vy measure, zero Gaussian variance and with support on $A$ is dense in the class of probability distributions with support on $A$ with respect to weak convergence.
\end{thm}
\begin{proof}
	Some arguments of the proof are in nature similar to the ones of the proof of Theorem 4.1 in \cite{LPS}. First, we prove the result when $A$ is bounded.
	
	Let $A$ be a finite closed interval, thus $A=[k,c]$ for some $k,c\in\mathbb{R}$. Let $\mu$ be a probability distribution with support $[k,c]$. For $n \in\mathbb{N}$, let $b_{j,n} =k+(c-k)j/2n^{2}$, $j\in\{0,...,2n^{2}\}$ and define the discrete distribution $\mu_{n}$ concentrated on the lattice $\{b_{0,n},...,b_{2n^{2},n} \}$ by
		\begin{equation}\label{eq-density}
		\mu_{n}(\{b_{j,n} \})=\begin{cases}
		\mu((-\infty,b_{0,n}]),& j=0,\\
		\mu((b_{j-1,n},b_{j,n}]),& j=1,...,2n^{2}-1,\\
		\mu((b_{2n^{2}-1,n},\infty)),& j=2n^{2}.\\
		\end{cases}
		\end{equation}
	Then, $\mu_{n}\stackrel{w}{\rightarrow}\mu$ as $n\rightarrow\infty$. Observe that $\mu_{n}$ is the probability distribution of a random variable with values on $\{b_{0,n},...,b_{2n^{2},n} \}\subset[k,c]$. It remains to prove that each $\mu_{n}$ is a weak limit of QID distributions with finite quasi-L\'{e}vy measure, zero Gaussian variance and with support on $[k,c]$. W.l.o.g.~assume that the approximating sequence of distributions $\sigma$ is such that $\sigma(\{b_{j,n} \})>0$ for every $j\in\{0,...,2n^{2}\}$. Assume that the characteristic function $\hat{\sigma}$ has zeros (in the other case we can directly use Corollary 3.10 in \cite{LPS} to conclude). Let $X$ be a random variable with distribution $\sigma$ and define $Y = \frac{(X-k)2n^{2}}{c-k}$. Then, $Y$ is concentrated on $\{0,...,2n^{2}\}$ with masses $a_{j} = \mathbb{P}(Y = j) > 0$ for $j = 0, . . . , 2n^{2}$, and its characteristic function has zeroes. Then, the polynomial $f (w) =\sum_{j=0}^{2n^{2}}a_{j}w^{j}$ has zeroes on the unit circle. Factorizing, we obtain $f(w)=a_{2n^{2}}\prod_{j=1}^{2n^{2}}(w-\xi_{j})$, where $\xi_{j}$, $j=1,...,2n^{2}$, denote the complex roots. Let $f_{h}(w)=a_{2n^{2}}\prod_{j=1}^{2n^{2}}(w-\xi_{j}-h)$, where $w\in\mathbb{C}$ and $h>0$. Then, for small enough $h$, $f_{h}$ is a polynomial with real coefficients, namely $f_{h}(w)=\sum_{j=0}^{2n^{2}}a_{h,j}w^{j}$ with $a_{h,j}\in\mathbb{R}$. Observe that for small enough $h$, $a_{h,j}$ and $a_{j}$ will be close, so $a_{h,j}>0$. Now, let $Z_{h}$ be a random variable with distribution $\sigma_{h}=\left(\sum_{j=0}^{2n^{2}}a_{h,j}\right)^{-1}\sum_{j=0}^{2n^{2}}a_{h,j}\delta_{j}$ and let $X_{h}=\frac{Z_{h}(c-k)}{2n^{2}}+k$. Observe that, for every $h>0$, $X_{h}$ is random variable with values on the lattice $\{b_{0,n},...,b_{2n^{2},n} \}$ and its characteristic function has no zeros, and that $X_{h}\stackrel{d}{\rightarrow}X$ as $h\searrow0$. Finally, by Corollary 3.10 in \cite{LPS} we know that $X_{h}$ is QID with finite quasi-L\'{e}vy measure and zero Gaussian variance.
	
	Observe that if $A$ is a bounded open interval, say $A=(k',c')$ for some $c,k\in\mathbb{R}$, then the above arguments apply. Let $\mu$ be a probability distribution with support $(k',c')$. For any $n\in\mathbb{N}$ let $k'_{n}=k'+\frac{(c'-k')}{2n^{2}}$ and $c'_{n}=c'-\frac{(c'-k')}{2n^{2}}$ and let $b_{j,n} =k'_{n}+(c'_{n}-k'_{n})j/2n^{2}$, $j\in\{0,...,2n^{2}\}$ and define the discrete distribution $\mu_{n}$ concentrated on the lattice $\{b_{0,n},...,b_{2n^{2},n} \}$ as in (\ref{eq-density}). Then, $\mu_{n}\stackrel{w}{\rightarrow}\mu$ as $n\rightarrow\infty$ and, applying the same reaming arguments (in which $n$ is fixed) for $k'_{n}$ and $c'_{n}$ instead of $k$ and $c$, we obtain the result for $A$ bounded and open.
	
	Let now $A$ be an unbounded interval of the form $A=[k,\infty)$ for some $k\in\mathbb{R}$. Let $\mu$ be a probability distribution with support on $[k,\infty)$. For $n \in\mathbb{N}$, let $b_{j,n} = k+j/n$, $j\in\{0,...,2n^{2}\}$ and define the discrete distribution $\mu_{n}$ concentrated on the lattice $\{b_{0,n},...,b_{2n^{2},n} \}$ as in (\ref{eq-density}). Then, $\mu_{n}\stackrel{w}{\rightarrow}\mu$ as $n\rightarrow\infty$. Using the notation above, let $X$ be a random variable with distribution $\sigma$ and define $Y = (X-k)n$. Then, $Y$ is concentrated on $\{0,...,2n^{2}\}$ with masses $a_{j}$ and its characteristic function has zeroes by assumption. We proceed as before. Thus, for small enough $h$, we obtain a polynomial with real coefficients $f_{h}$, namely $f_{h}(w)=\sum_{j=0}^{2n^{2}}a_{h,j}w^{j}$ with $a_{h,j}\in\mathbb{R}$ and $a_{h,j}>0$, for small enough $h$. Then, let $Z_{h}$ be a random variable with distribution $\sigma_{h}=\left(\sum_{j=0}^{2n^{2}}a_{h,j}\right)^{-1}\sum_{j=0}^{2n^{2}}a_{h,j}\delta_{j}$ and let $X_{h}=\frac{Z_{h}}{n}+k$. Then, $X_{h}$ is random variables with support on $\{b_{0,n},...,b_{2n^{2},n} \}\subset[k,\infty)$ and its characteristic function has no zeros, and that $X_{h}\stackrel{d}{\rightarrow}X$ as $h\searrow0$. Hence, by Corollary 3.10 in \cite{LPS} we obtain the result. 
	
	Similarly we obtain the result for $(k',\infty)$, for $(-\infty,c]$ and for $(-\infty,c')$, where $k',c,c'\in\mathbb{R}$.
\end{proof}

Recall that the L\'{e}vy-Prokhorov metric (or better just L\'{e}vy metric since we work on $\mathbb{R}$) for two probability distributions $F$ and $G$ on $\mathbb{R}$ is defined as
\begin{equation*}
\rho(F,G):=\inf\left\{\varepsilon>0\,|\,F(x-\varepsilon)-\varepsilon\leq G(x)\leq F(x+\varepsilon)+\varepsilon\textnormal{ for all $x\in\mathbb{R}$} \right\}.
\end{equation*}
\begin{lem}\label{lemmaProk}
Let $F$ and $G$ be any two probability distributions on $\mathbb{R}$ and let $F_{c}(x):=F(\frac{x}{c})$ and $G_{c}(x):=G(\frac{x}{c})$ where $c\in\mathbb{R}\setminus\{0\}$. For every positive constant $c\leq1$ we have that $\rho(F_{c},G_{c})\leq \rho(F,G)$.
\end{lem}
\begin{proof}
	Let $c$ be any positive constant $c\leq1$. Observe that $F_{c}(x-\varepsilon)=F(\frac{x-\varepsilon}{c})\leq F(\frac{x}{c}-\varepsilon)$ and similarly we have that $F_{c}(x+\varepsilon)\geq F(\frac{x}{c}+\varepsilon)$. This implies that if $\varepsilon>0$ satisfies $F(x-\varepsilon)-\varepsilon\leq G(x)\leq F(x+\varepsilon)+\varepsilon\textnormal{ for all $x\in\mathbb{R}$}$, then it also satisfies $F_{c}(x-\varepsilon)-\varepsilon\leq G_{c}(x)\leq F_{c}(x+\varepsilon)+\varepsilon\textnormal{ for all $x\in\mathbb{R}$}$. Then, we have
\begin{equation*}
\rho(F_{c},G_{c})=\inf\left\{\varepsilon>0\,|\,F_{c}(x-\varepsilon)-\varepsilon\leq G_{c}(x)\leq F_{c}(x+\varepsilon)+\varepsilon\textnormal{ for all $x\in\mathbb{R}$} \right\}
\end{equation*}
\begin{equation*}
\leq\inf\left\{\varepsilon>0\,|\,F(x-\varepsilon)-\varepsilon\leq G(x)\leq F(x+\varepsilon)+\varepsilon\textnormal{ for all $x\in\mathbb{R}$} \right\}=\rho(F,G).
\end{equation*}
\end{proof}
Observe that for two real valued random variables $X$ and $Y$ the above lemma affirms that for any $0<c\leq1$ we have that $\rho(cX,cY)\leq\rho(X,Y)$.

Another useful property of the Prokhorov metric is the following. From condition 3) of the section ``L\'{e}vy metric" in \cite{Haz} (page 405) given any probability distributions on $\mathbb{R}$ $F_{1},...,F_{k},G_{1},...,G_{k}$, where $k\in\mathbb{N}$, we have that
\begin{equation}\label{Prok}
\rho(F_{1}*\cdots*F_{k},G_{1}*\cdots*G_{k})\leq \sum_{j=1}^{k}\rho(F_{j},G_{j}).
\end{equation}
For the next two results denote by $S_{n}$ the sequence of bounded sets (\textit{i.e.}~$S_{n}\in\hat{\textbf{S}}$) s.t.~$S_{n}\uparrow S$. Notice that such sequence exists by the definition of $\hat{\textbf{S}}$, see page 19 in \cite{Kallenberg2}.
\begin{pro}\label{pro-Poisson}
Consider an atomless CRM $\alpha$ with corresponding unique pair $(\gamma,F)$. Let $\gamma_{n}(A)=\gamma(S_{n}\cap A)$ and let $F_{n}(C)=F(C\cap(S_{n}\times(\frac{1}{n},\infty)))$, for every $A\in\textbf{S}$, $C\in\textbf{S}\otimes\mathcal{B}((0,\infty))$ and $n\in\mathbb{N}$. Then, $\gamma_{n}$ and $F_{n}$ are finite measures and there exists a sequence of atomless finite CRMs $\alpha_{n}$ with pair $(\gamma_{n},F_{n})$ s.t.~$\alpha_{n}\stackrel{d}{\to}\alpha$. 
\end{pro}
\begin{proof}
From Kingman's representation theorem (see \cite{King} and see also Corollary 12.11 in \cite{Kallenberg} and Corollary 3.21 in \cite{Kallenberg2}), we have that every atomless CRM $\alpha$ has the following representation:
\begin{equation}\label{King-alpha}
\alpha=\gamma+\int_{0}^{\infty}\int_{S}x\delta_{s}\eta(ds\, dx),\quad \textnormal{a.s.}
\end{equation}
for some non-random diffuse measure $\gamma\in\mathcal{M}_{S}$ and a Poisson process $\eta$ on $S\times(0,\infty)$ with intensity $F$ satisfying
\begin{equation}\label{Poisson}
\int_{0}^{\infty}(1\wedge x)F(A\times dx)<\infty,
\end{equation}
for every $A\in\hat{\textbf{S}}$. In particular, for every $B\in\textbf{S}$ we have that $\alpha(B)<\infty$ if and only if $\gamma(B)<\infty$ and condition $(\ref{Poisson})$ holds for $B\in\textbf{S}$ (see Corollary 12.11 in \cite{Kallenberg}). Further, notice that the above formulation implies that for every $A\in\textbf{S}$ and $f\in\hat{C}_{S}$
\begin{equation*}
\alpha(A)=\gamma(A)+\int_{0}^{\infty}x\eta(A\times dx)\,\,\,\textnormal{and }\,\,\, \alpha f=\gamma f+\int_{0}^{\infty}\int_{S}xf(s)\eta(ds\, dx),\,\,\, \textnormal{a.s..}
\end{equation*}
Moreover, the unique one to one correspondence between $\alpha$ and $(\gamma, F)$ is shown in Theorem 3.20 of \cite{Kallenberg2}. 

It is possible to see that $\gamma_{n}$ and $F_{n}$ are measures on $\textbf{S}$ and on $\textbf{S}\otimes\mathcal{B}((0,\infty))$, respectively. In particular, since $\alpha(A)<\infty$ for every $A\in\hat{\textbf{S}}$ then $\gamma(S_{n})<\infty$ and 
\begin{equation*}
\int_{0}^{\infty}(1\wedge x)F(S_{n}\times dx)<\infty\Rightarrow F(S_{n}\times (\frac{1}{n},\infty))<\infty,
\end{equation*}
for every $n\in\mathbb{N}$. Thus, $\gamma_{n}$ and $F_{n}$ are finite measures, for every $n\in\mathbb{N}$. 

Now, for every $n\in\mathbb{N}$, let $\eta_{n}$ be a Poisson process on $S\times(0,\infty)$ with intensity $F_{n}$ and let
\begin{equation*}
\alpha_{n}=\gamma_{n}+\int_{0}^{\infty}\int_{S}x\delta_{s}\eta_{n}(ds\, dx).
\end{equation*} 
Then, we have that $\alpha_{n}$ is an atomless CRM and since $\gamma_{n}$ and $F_{n}$ are finite then $\alpha_{n}$ is finite, for every $n\in\mathbb{N}$ (see Corollary 12.11 in \cite{Kallenberg}).
 
Concerning the stated convergence we have the following. From Lemma 12.2 in \cite{Kallenberg} (or from Lemma 3.1 in \cite{Kallenberg2}) we have that for every $f\in\hat{C}_{S}$
\begin{equation*}
-\log\mathbb{E}\left[\exp\left(-\int f(s)\alpha(ds)\right)\right]=\gamma f+\int_{0}^{\infty}\int_{S}1-e^{-x\delta_{s}f}F(ds\, dx).
\end{equation*}
Hence, by assumption we have that for every $f\in\hat{C}_{S}$
\begin{equation*}
-\log\mathbb{E}\left[\exp\left(-\int f(s)\alpha(ds)\right)\right]+\log\mathbb{E}\left[\exp\left(-\int f(s)\alpha_{n}(ds)\right)\right]
\end{equation*}
\begin{equation*}
=\int_{S}f(s)\gamma(ds)+\int_{0}^{\infty}\int_{S}1-e^{-xf(s)}F(ds\, dx)-\int_{S_{n}}f(s)\gamma(ds)+\int_{\frac{1}{n}}^{\infty}\int_{S_{n}}1-e^{-xf(s)}F(ds\, dx)
\end{equation*}
\begin{equation*}
=\int_{S\setminus S_{n}}f(s)\gamma(ds)+\int_{0}^{\frac{1}{n}}\int_{S_{n}}1-e^{-xf(s)}F(ds\, dx)+\int_{0}^{\infty}\int_{S\setminus S_{n}}1-e^{-xf(s)}F(ds\, dx)\to 0,
\end{equation*}
as $n\to\infty$. Then, by point (iii) in Theorem \ref{density-T1} (see also Lemma 4.24 in \cite{Kallenberg2}) we obtain that $\alpha_{n}\stackrel{d}{\to}\alpha$, as $n\to\infty$.
\end{proof}
Now, let us denote by $\mathcal{I}$ the set of all CRMs on $S$ (considered as random elements in $\mathcal{M}_{S}$ endowed with the vague topology) and recall that $\mathbb{Z}_{+}=\mathbb{N}\cup \{0\}$. From Theorem 7.1 in \cite{Kallenberg0} we know that an element $\xi$ of $\mathcal{I}$ has the following representation
\begin{equation*}
\xi\stackrel{a.s.}{=}\alpha+\sum_{j=1}^{K}\beta_{j}\delta_{s_{j}}
\end{equation*}
with $K\in\mathbb{Z}_{+}\cup \{\infty\}$, where $\{s_{j}:j\geq1\}$ is the set of fixed atoms of $\xi$ in $S$, $\alpha$ is an atomless CRM, and $\beta_{j}$, $j\geq1$, are $\mathbb{R}_{+}$-valued random variables, which are mutually independent and independent of $\alpha$. We call $\sum_{j=1}^{K}\beta_{j}\delta_{s_{j}}$ the fixed component of $\xi$. We remark that in the Kingman's representation $\alpha$ is the sum of a deterministic and a ordinary component as shown in the proof of Proposition \ref{pro-Poisson} in eq.~(\ref{King-alpha}).
 
Consider the following class of QID random measures:
\begin{equation*}
\mathcal{A}:=\bigg\{\xi\in\mathcal{I}\bigg|\xi\stackrel{a.s.}{=}\alpha+\sum_{j=1}^{K}\beta_{j}\delta_{s_{j}},\textnormal{with $\alpha$ an atomless CRM with finite L\'{e}vy measure, $\{s_{j}:j=1,...,K\}$}  
\end{equation*}
\begin{equation*}
\textnormal{a finite set of fixed atoms in $S$, and $\beta_{j}$, $j\geq1$, $\mathbb{R}_{+}$-valued QID random variables with finite quasi-L\'{e}vy}
\end{equation*}
\begin{equation*}
\textnormal{measure and zero Gaussian variance and that are mutually independent and independent of $\alpha$}\bigg\}.
\end{equation*}
First, notice that $\alpha$ is ID, because any atomless random measure with independent increments is ID. Second, observe that, in contrast with the usual representation of CRMs, the elements of $\mathcal{A}$ have that the atomless random measure $\alpha$ has finite L\'{e}vy measure (thus, $\alpha$ is finite), that the number of fixed atoms $K$ is finite, and that $\beta_{j}$, $j=1,...,K$, $\mathbb{R}_{+}$-valued QID random variables with finite quasi-L\'{e}vy measure and zero Gaussian variance. Notice that the elements of $\mathcal{A}$ are almost surely finite on $\textbf{S}$. Thus, $\mathcal{A}$ is strictly smaller than the class of QID CRMs, which in turn is strictly smaller than the class of all CRMs (namely $\mathcal{I}$).

We are ready to present the main result of this section.
\begin{thm}\label{density-T2}$\mathcal{A}$ is dense in the space of all CRMs with respect to the convergence in distribution.
\end{thm}
\begin{proof}
From Theorem 7.1 in \cite{Kallenberg0} we know that any CRM has the following unique representation
\begin{equation}\label{measure1}
\xi\stackrel{a.s.}{=}\alpha+\sum_{j=1}^{K}\beta_{j}\delta_{s_{j}}
\end{equation}
with $K\leq\infty$, where $\{s_{j}:j\geq1\}$ is the set of fixed atoms of $\xi$, $\alpha$ is a random measure without fixed atoms with independent increments (hence, $\alpha$ is an atomless ID random measure), and $\beta_{j}$, $j\geq1$, are $\mathbb{R}_{+}$-valued random variables, which are mutually independent and independent of $\alpha$.

From Theorem \ref{pro8} with $A=[0,\infty)$, we know that for each $\beta_{j}$ there exists a sequence of non-negative QID random variable with zero Gaussian variance and finite L\'{e}vy measure that converges in distribution to $\beta_{j}$, for every $j\in\mathbb{N}$. Denote by $\beta_{n,j}$ such a sequence.

Denote by $S_{n}$ the sequence of bounded sets s.t.~$S_{n}\uparrow S$ and by $(\gamma,F)$ be the pair associated to $\alpha$. Let $\gamma_{n}(A)=\gamma(S_{n}\cap A)$ and  $F_{n}(C)=F(C\cap(S_{n}\times(\frac{1}{n},\infty)))$, for every $A\in\textbf{S}$, $C\in\textbf{S}\otimes\mathcal{B}((0,\infty))$ and $n\in\mathbb{N}$, as in Proposition \ref{pro-Poisson}. Then, by Proposition \ref{pro-Poisson} there exists a sequence of finite CRMs $\alpha_{n}$ with pair $(\gamma_{n},F_{n})$ s.t.~$\alpha_{n}\stackrel{d}{\to}\alpha$.

The first step is to show the existence of random measures $\xi_{n}\in\mathcal{A}$ with ID random measure equals in distribution to $\alpha_{n}$, with fixed atoms in $\{s_{j}:j\geq1\}$ and weights equal in distributions to $\beta_{n,j}$. The existence is not immediate because we do not know whether the $\beta_{n,j}$ are mutually independent and independent of $\alpha_{n}$ in the underlying probability space of $\xi$. This is a classical problem in probability and the solution lies in the construction of a probability space under which these conditions are satisfied, which is given by the `product' of the probability spaces.

For the sake of clarity and completeness let us write here the arguments. Fix $n\in\mathbb{N}$. Denote the underlying probability spaces of $\alpha_{n}$ by $(\Omega,\mathcal{F},\mathbb{P})$ and of the random variable $\beta_{n,j}$ by $(\Omega_{j},\mathcal{F}_{j},\mathbb{P}_{j})$, for $j=1,...,n$. Consider the probability space $(\Omega',\mathcal{F}',\mathbb{P}')$ where $\Omega'=\Omega\times\Omega_{1}\times\cdots\times\Omega_{n}$, $\mathcal{F}'=\mathcal{F}\otimes\mathcal{F}_{1}\otimes\cdots\otimes \mathcal{F}_{n}$ and $\mathbb{P}'$ is the product probability measure of $\mathbb{P},\mathbb{P}_{1}$,...,$\mathbb{P}_{n}$.

Let $\alpha_{n}'(\cdot)(\omega,\omega_{1},...,\omega_{n}):=\alpha_{n}(\cdot)(\omega)$ and let $\beta'_{n,j}(\omega,\omega_{1},...,\omega_{n}):=\beta_{n,j}(\omega_{j})$, where $j=1,...,n$, for every $(\omega,\omega_{1},...,\omega_{n})\in\Omega'$. Observe that for every $B_{1},...,B_{k}\in\textbf{S}$ and $x_{1},...,x_{k},x^{(1)}_{1},...,x^{(1)}_{k},...,x^{(n)}_{1},...,x^{(n)}_{k}\in\mathbb{R}_{+}$ we have that 
\begin{equation*}
\mathbb{P}'\Big(\alpha_{n}'(B_{1})<x_{1},...,\alpha_{n}'(B_{k})<x_{k},\delta_{s_{1}}(B_{1})\beta'_{n,1}<x^{(1)}_{1},...,\delta_{s_{1}}(B_{k})\beta'_{n,1}<x^{(1)}_{k},
\end{equation*}
\begin{equation*}
..., \delta_{s_{n}}(B_{1})\beta'_{n,n}<x^{(n)}_{1},...,\delta_{s_{n}}(B_{k})\beta'_{n,n}<x^{(n)}_{k}\Big)=\mathbb{P}\Big(\alpha(B_{1})<x_{1},...,\alpha(B_{k})<x_{k}\Big)
\end{equation*}
\begin{equation*}
\mathbb{P}_{1}\Big(\delta_{s_{1}}(B_{1})\beta'_{n,1}<x^{(1)}_{1},...,\delta_{s_{1}}(B_{k})\beta'_{n,1}<x^{(1)}_{k} \Big)
\cdots\mathbb{P}_{n}\Big(\delta_{s_{n}}(B_{1})\beta'_{n,n}<x^{(n)}_{1},...,\delta_{s_{n}}(B_{k})\beta'_{n,n}<x^{(n)}_{k} \Big)
\end{equation*}
Now, let 
\begin{equation}\label{Indi}
\xi_{n}(\cdot)(\omega'):=\alpha_{n}'(\cdot)(\omega')+\sum_{j=1}^{n}\beta'_{n,j}(\omega')\delta_{s_{j}}(\cdot),\quad\forall\omega'\in\Omega',
\end{equation}
where $s_{1}$,..., $s_{n}$ are the same as the ones in (\ref{measure1}). It is possible to see that, for every $\omega'\in\Omega'$, $\xi_{n}(\cdot)(\omega')$ is a measure because it is the sum of measures and that, for every $B\in\textbf{S}$, $\xi_{n}(B)(\cdot)$ is a measurable function because it is the sum of measurable functions. Thus, $\xi_{n}$ is a random measure on $S$ and from its definition it is possible to see that it belongs to $\mathcal{A}$. 

Since $\beta_{n,j}\stackrel{d}{\rightarrow}\beta_{j}$ we can choose a subsequence of $\beta_{n,j}$, which by abuse of notation we denote it by $\beta_{n,j}$, such that $\rho(\beta_{n,j},\beta_{j})<\frac{1}{n^{2}}$ for every $j=1,...,n$ and  $n\in\mathbb{N}$. From the above arguments there exists a sequence of random measures in $\mathcal{A}$ (with possibly different underlying probability spaces) such that $\xi_{n}=\alpha_{n}'+\sum_{j=1}^{n}\beta'_{n,j}\delta_{s_{j}}$. Thus, using that $\beta'_{n,j}\stackrel{d}{=}\beta_{n,j}$ we obtain that $\rho(\beta'_{n,j},\beta_{j})<\frac{1}{n^{2}}$ for every $j=1,...,n$ and  $n\in\mathbb{N}$.

Now, we need to show that $\xi_{n}\stackrel{vd}{\rightarrow}\xi$. From Theorem \ref{density-T1}, it is sufficient to show that $\int f(x)\xi_{n}(dx)\stackrel{d}{\rightarrow}\int f(x)\xi(dx)$ for all $f\in\hat{C}_{S}$. Since $\alpha'_{n}\stackrel{d}{=}\alpha_{n}$ for every $n\in\mathbb{N}$ and $\alpha_{n}\stackrel{d}{\to}\alpha$ for every $\omega\in\Omega$ then $\alpha'_{n}\stackrel{d}{\to}\alpha$. Further, since $\alpha'_{n}$ and $\alpha$ are independent of the corresponding fixed component, this reduces the goal to prove that $\sum_{j=1}^{n} f(s_{j})\beta_{n,j}'\stackrel{d}{\rightarrow}\sum_{j=1}^{\infty} f(s_{j})\beta_{j}$ for all $f\in\hat{C}_{S}$. 

Let $f\in\hat{C}_{S}$, hence, $f$ is bounded and has bounded support, and by denoting $B$ the support of $f$ we have that $B\in\hat{\textbf{S}}$ and so that almost surely $\xi_{n}(B)<\infty$, $n\in\mathbb{N}$, and $\xi(B)<\infty$. Thus, for each $n\in\mathbb{N}$, $\sum_{j=1}^{n} f(s_{j})\beta_{n,j}'<\infty$ a.s.~and $\sum_{j=1}^{\infty} f(s_{j})\beta_{j}<\infty$ a.s..

Moreover, notice that it is sufficient to prove the result for any $f\in \hat{C}_{S}$ with $f(s)\leq1$ for every $s\in S$. Indeed, consider any $f\in\hat{C}_{S}$ and let $\bar{C}\in\mathbb{R}_{+}$ be its bound, then $\sum_{j=1}^{n} f(s_{j})\beta_{n,j}'=\bar{C}\sum_{j=1}^{n} \frac{f(s_{j})}{\bar{C}}\beta_{n,j}'$ and so if $\sum_{j=1}^{n} \frac{f(s_{j})}{\bar{C}}\beta_{n,j}'\stackrel{d}{\rightarrow}\sum_{j=1}^{\infty} \frac{f(s_{j})}{\bar{C}}\beta_{j}$ then $\sum_{j=1}^{n} f(s_{j})\beta_{n,j}'\stackrel{d}{\rightarrow}\sum_{j=1}^{\infty} f(s_{j})\beta_{j}$.

Now, consider any $f\in \hat{C}_{S}$ with $f(s)\leq1$ for every $s\in S$. By the triangular inequality we have that 
\begin{equation*}
\rho\left(\sum_{j=1}^{n}f(s_{j})\beta_{n,j}',\sum_{j=1}^{\infty}f(s_{j})\beta_{j}\right)
\leq \rho\left(\sum_{j=1}^{n}f(s_{j})\beta_{n,j}',\sum_{j=1}^{n}f(s_{j})\beta_{j}\right)+\rho\left(\sum_{j=1}^{n}f(s_{j})\beta_{j},\sum_{j=1}^{\infty}f(s_{j})\beta_{j}\right).
\end{equation*}
The last element converges to zero as $n\rightarrow\infty$ because $\sum_{j=1}^{n}f(s_{j})\beta_{j}\stackrel{a.s.}{\rightarrow}\sum_{j=1}^{\infty}f(s_{j})\beta_{j}$ as $n\rightarrow\infty$. For the other element, by (\ref{Prok}) and by Lemma \ref{lemmaProk} we obtain that
\begin{equation*}
\rho\left(\sum_{j=1}^{n}f(s_{j})\beta_{n,j}',\sum_{j=1}^{n}f(s_{j})\beta_{j}\right)\leq \sum_{j=1}^{n}\rho\left(f(s_{j})\beta_{n,j}',f(s_{j})\beta_{j}\right)\leq \sum_{j=1}^{n}\rho\left(\beta_{n,j}',\beta_{j}\right)<\frac{1}{n}.
\end{equation*}
Thus, we have that $\sum_{j=1}^{n} f(s_{j})\beta_{n,j}'\stackrel{d}{\rightarrow}\sum_{j=1}^{\infty} f(s_{j})\beta_{j}$ as $n\rightarrow\infty$, which concludes the proof.
\end{proof}
\begin{rem}
	\textnormal{We could alternatively consider an almost sure equality in (\ref{Indi}) and then use the existence and uniqueness results for random measures (see Theorem 2.15 and Corollary 2.16 in \cite{Kallenberg}) to obtain a random measure almost surely equal to $\xi_{n}$. In addition, by the Kolmogorov extension theorem the same arguments of the first part of the above proof hold for the case of $n$ `equal' to infinity, namely $\xi_{n}=\alpha_{n}'+\sum_{j=1}^{\infty}\beta'_{n,j}$.
	\\
	Further, we point out that if $\xi$ is such that the number of fixed atoms in any bounded set (\textit{i.e.}~in any $B\in\hat{\textbf{S}}$) is finite then the number of fixed atoms in the support of every $f\in\hat{C}_{S}$ is finite, namely $\{s_{j}:j\geq1\}\cap\textnormal{supp}(f)$ has finite cardinality, and so the stated result follows directly from the mutual independence of the $\beta_{n,j}'$, $j=1,...,n$, from the fact that $\beta_{n,j}'\stackrel{d}{\rightarrow}\beta_{j}$ as $n\rightarrow\infty$, for every $j=1,...,n$ and $n\in\mathbb{N}$, and from the continuous mapping theorem.}
\end{rem}
\begin{rem}\label{rem-alpha-infinity}
	\textnormal{Let $\mathcal{A}_{\infty}$ be a class of random measures like $\mathcal{A}$, but such that the ID component is not necessarily finite, \textit{i.e.}~the `$\alpha$' is not necessarily finite. Then, trivially $\mathcal{A}_{\infty}$ is dense in $\mathcal{I}$ w.r.t.~the convergence in distribution. Indeed, let $\xi=\alpha+\sum_{j=1}^{K}\beta_{j}\delta_{s_{j}}$ be any CRM on $S$. If we know the ID component of $\xi$, \textit{i.e.}~$\alpha$, and for modelling/theoretical reasons we can take an approximating sequence of \textit{unbounded} $\xi_{n}$, then we can define the $\xi_{n}$ s.t.~$\xi_{n}(\cdot)(\omega'):=\tilde{\alpha}_{n}'(\cdot)(\omega')+\sum_{j=1}^{n}\beta'_{n,j}(\omega')\delta_{s_{j}}(\cdot),$ $\forall\omega'\in\Omega'$, where $\tilde{\alpha}_{n}'(\cdot)(\omega,\omega_{1},...,\omega_{n}):=\alpha(\cdot)(\omega)$. Then, $\xi_{n}\in\mathcal{A}_{\infty}$ and from the arguments of the proof of Theorem \ref{density-T2} it is possible to see that $\xi_{n}\stackrel{d}{\to}\xi$.}
\end{rem}
It is possible to consider also the set of bounded measures, denoted by $\hat{\mathcal{M}}_{S}$, which can be endowed with the vague topology, as for $\mathcal{M}_{S}$, but also with the weak topology. The weak topology on $\hat{\mathcal{M}}_{S}$ is the topology generated by the integration maps $\pi_{f}$ for all bounded continuous functions. Then, for random measures $\xi,\xi_{1},\xi_{2},...$ considered as random elements in $\hat{\mathcal{M}}_{S}$, endowed with the weak topology, we will denote by $\xi_{n}\stackrel{wd}{\rightarrow}\xi$ the convergence in distribution. Observe that in this setting a QID random measures as defined in Definition are QID random measures on $(S,\textbf{S})$ (hence we do not need to extend them) because for every $B\in\textbf{S}$ they are all a.s.~bounded.

We will use the following result of Kallenberg to prove our next result.
\begin{thm}[see Theorem 4.19 in \cite{Kallenberg2}]\label{density-T"}
	Let $\xi,\xi_{1},\xi_{2},...$ be a.s.~bounded random measures on $S$. Then these conditions are equivalent
	\\\textnormal{(i)} $\xi_{n}\stackrel{wd}{\rightarrow}\xi$,
	\\\textnormal{(ii)} $\xi_{n}\stackrel{vd}{\rightarrow}\xi$, and $\xi_{n}(S)\stackrel{d}{\rightarrow}\xi(S)$.
\end{thm}
We are now ready to present our next result, which is similar to Theorem \ref{density-T2}, but applies to $\hat{\mathcal{M}}_{S}$ and involves both the vague and the weak topology.
\begin{thm}\label{theorem-bounded}$\mathcal{A}$ is dense in the space of all CRMs, considered as random elements in $\hat{\mathcal{M}}_{S}$, endowed with either the vague topology or the weak topology, with respect to the convergence in distribution.
\end{thm}
\begin{proof}
Consider first the case of $\hat{\mathcal{M}}_{S}$ endowed with the vague topology. Then, by the same arguments as the ones used in the proof of Theorem \ref{density-T2} we obtain the result.

For the weak topology case, by the same arguments as the ones used in the proof of Theorem \ref{density-T2} we have that $\xi_{n}\stackrel{vd}{\rightarrow}\xi$. Hence, according to Theorem \ref{density-T"} it remains to prove that $\xi_{n}(S)\stackrel{d}{\rightarrow}\xi(S)$, namely that $\alpha_{n}'(S)+\sum_{j=1}^{n}\beta'_{n,j}\stackrel{d}{\rightarrow}\alpha(S)+\sum_{j=1}^{\infty}\beta_{j}$. However, this has been proved in the proof of Theorem \ref{density-T2} -- indeed, consider $f\equiv1$ and notice that $\xi_{n}(S)$ and $\xi(S)$ are a.s.~finite since $\xi_{n}$ and $\xi$ are almost surely bounded. Thus, the proof is complete.
\end{proof}
\subsection{The density result for QID point processes}\label{Subsec-Density-Point}
In this subsection we answer positively the following question: given any point process with independent increments is it possible to find a sequence of QID point processes with independent increments which converges in distribution to it?

Thus, in this subsection we restrict our focus to point processes with independent increments and check that the density result holds. There are two main reasons for doing this. First, the class of point processes with independent increments represents one of the most studied class of completely random measures due to their nice theoretical properties and their importance in applications. Second, we have an explicit formulation for the quasi-L\'{e}vy measure and the drift of QID random variables supported on finite subsets of $\mathbb{Z}_{+}$ (see Theorem 3.9 in \cite{LPS}).

Let us first show the density result for random variables supported on $\mathbb{Z}_{+}$.
\begin{pro}\label{pro9}
	The class of QID distributions supported on finite subsets of $\mathbb{Z}_{+}$ is dense in the class of probability distributions with support on $\mathbb{Z}_{+}$ with respect to weak convergence.
	\end{pro}
	\begin{proof}
Let $\mu$ be a probability distribution with support on $\mathbb{N}$. For $n \in\mathbb{N}$, define the discrete distribution $\mu_{n}$ concentrated on the lattice $\{0,...,2n \}$ by
		\begin{equation*}
		\mu_{n}(\{j \})=\mu(\{j\}),\quad j\in\{0,...,2n\}.
		\end{equation*}
		Then, $\mu_{n}\stackrel{w}{\rightarrow}\mu$ as $n\rightarrow\infty$. It remains to prove that each $\mu_{n}$ is a weak limit of QID distributions with support on $\{0,...,2n \}$. 
		W.l.o.g.~assume that the approximating sequence of distributions $\sigma$ is such that $\sigma(\{j \})>0$ for every $j\in\{0,...,2n\}$. Assume that the characteristic function $\hat{\sigma}$ has zeros (in the other case we can directly use Theorem 3.9 in \cite{LPS} to conclude). Let $X$ be a random variable with distribution $\sigma$ and let $a_{j} = \mathbb{P}(X = j) > 0$ for $j = 0, . . . , 2n$. Then, the polynomial $f (w) =\sum_{j=0}^{2n}a_{j}w^{j}$ has zeroes on the unit circle. Factorizing, we obtain $f(w)=a_{2n}\prod_{j=1}^{2n}(w-\xi_{j})$, where $\xi_{j}$, $j=1,...,2n$, denote the complex roots. Let $f_{h}(w)=a_{2n}\prod_{j=1}^{2n}(w-\xi_{j}-h)$, where $w\in\mathbb{C}$ and $h>0$. Then, for small enough $h$, $f_{h}$ is a polynomial with real coefficients, namely $f_{h}(w)=\sum_{j=0}^{2n}a_{h,j}w^{j}$ with $a_{h,j}\in\mathbb{R}$. Observe that for small enough $h$, $a_{h,j}$ and $a_{j}$ will be close, so $a_{h,j}>0$. Now, let $X_{h}$ be a random variable with distribution $\sigma_{h}=\left(\sum_{j=0}^{2n}a_{h,j}\right)^{-1}\sum_{j=0}^{2n}a_{h,j}\delta_{j}$. We conclude by noticing that, for every $h>0$, $X_{h}$ is random variable with values on the lattice $\{0,...,2n \}$ and its characteristic function has no zeros (thus it is QID by Theorem 3.9 in \cite{LPS}), and that $X_{h}\stackrel{d}{\rightarrow}X$ as $h\searrow0$. 
		\end{proof}
From Corollary 3.21 in \cite{Kallenberg2}, for an atomless point process with independent increments the corresponding unique pair, which we denote by $(\gamma,F)$, is such that $\gamma=0$ and $F$ is restricted to $S\times\mathbb{N}$.

Let $\mathcal{A}'$ be the set of all the point processes in $\mathcal{A}$. In other words, let
\begin{equation*}
\mathcal{A'}:=\bigg\{\xi\in\mathcal{I}\bigg|\xi\stackrel{a.s.}{=}\alpha+\sum_{j=1}^{K}\beta_{j}\delta_{s_{j}},\textnormal{with $\alpha$ an atomless point process with independent increments and}  
\end{equation*}
\begin{equation*}
\textnormal{finite L\'{e}vy measure, $\{s_{j}:j=1,...,K\}$ a finite set of fixed atoms in $S$, and $\beta_{j}$, $j\geq1$, QID random variables}
\end{equation*}
\begin{equation*}
\textnormal{concentrated on finite subsets of $\mathbb{Z}_{+}$ and that are mutually independent and independent of $\alpha$}\bigg\}.
\end{equation*}
Obviously, we have $\mathcal{A}'\subsetneq \mathcal{A}\subsetneq \mathcal{I}$.
\begin{thm}\label{density-T2-9}
$\mathcal{A}'$ is dense in the space of all point processes with independent increments with respect to the convergence in distribution.	
\end{thm}
\begin{proof}
	It follows from the same arguments as the ones used in the proof of Theorem \ref{density-T2}. In particular, now we need to use Proposition \ref{pro9} instead of Theorem \ref{pro8}. Further, now $\gamma_{n}=\gamma=0$ and $F_{n}$ and $F$ are concentrated on $S\times\mathbb{N}$. Then, following the same arguments as the ones used in the proof of Theorem \ref{density-T2} we obtain the result.
\end{proof}
We conclude this subsection with the density result for finite point processes (for which the weak topology might also be used), namely the equivalent of Theorem \ref{theorem-bounded} for point processes with independent increments.
\begin{pro}
	$\mathcal{A}'$ is dense in the space of point processes with independent increments, considered as random elements in $\hat{\mathcal{M}}_{S}$, endowed with either the vague topology or with the weak topology, with respect to the convergence in distribution.
\end{pro}
\begin{proof}
	It follows from the same arguments as the ones used in the proof of Theorem \ref{theorem-bounded}, with Theorem \ref{density-T2-9} instead of Theorem \ref{density-T2}.
\end{proof}
\section{Properties of the dense class $\mathcal{A}$}\label{Sec-Properties}
In this section we explore some of the properties of the random measures in $\mathcal{A}$, with a particular focus on spectral representations.

Consider the same notation as in the previous section. Let $\alpha$ be an atomless CRM (hence, ID). Using Theorem 12.10 and Corollary 12.11 in \cite{Kallenberg} we have that
\begin{equation}\label{1}
\hat{\mathcal{L}}(\alpha(A))(\theta)=\exp\left(-i\theta \gamma^{(1)}(A)+\int_{0}^{\infty}e^{i\theta x}-1F^{(1)}_{A}(dx)\right),
\end{equation}
for every $\theta\in\mathbb{R}$ and $A\in\textbf{S}$, where $\gamma$ is a finite diffuse measure on $\textbf{S}$ and $F$ is a finite measure on $\textbf{S}\otimes\mathcal{B}((0,\infty))$ with diffuse projections onto $S$. Observe that we can extend $F^{(1)}$ to a finite measure on $\textbf{S}\otimes\mathcal{B}(\mathbb{R})$, by assigning value zero outside $\textbf{S}\otimes\mathcal{B}((0,\infty))$; by abuse of notation, we call this finite measure $F^{(1)}$.

Further, let $\sum_{j=1}^{n}\delta_{s_{j}}\beta_{j}$, where $n\in\mathbb{N}$, $s_{j}\in S$, $j=1,...,n$, and where the $\beta_{j}$'s are mutually independent QID random variables with finite quasi-L\'{e}vy measure and zero Gaussian variance. With centering function equal zero (as in (\ref{1})), denote by $c_{j}$ and $b_{j}$ the drift and the quasi-L\'{e}vy measure of $\beta_{j}$, for $j=1,...,n$. Notice that we can use such centering function because the $\beta_{j}$'s have finite quasi L\'{e}vy measure. Then, the L\'{e}vy-Khintchine formulation of $\sum_{j=1}^{n}\delta_{s_{j}}\beta_{j}$ is given by
\begin{equation*}
\hat{\mathcal{L}}\bigg(\sum_{j=1}^{n}\delta_{s_{j}}(A)\beta_{j}\bigg)(\theta)=\exp\left(i\theta \gamma^{(2)}(A)+\int_{\mathbb{R}}e^{i\theta x}-1 F_{A}^{(2)}(dx)\right),
\end{equation*}
for every $\theta\in\mathbb{R}$ and $A\in\textbf{S}$, where $\gamma^{(2)}(A)=\sum_{j=1}^{n}\delta_{s_{j}}(A)c_{j}$ and $F_{A}^{(2)}(\cdot)=\sum_{j=1}^{n}\delta_{s_{j}}(A)b_{j}(\cdot)$. Then, $\xi=\alpha+\sum_{j=1}^{n}\delta_{s_{j}}\beta_{j}$ has the following formulation
\begin{equation}\label{2}
\hat{\mathcal{L}}(\xi(A))(\theta)=\exp\left(i\theta \nu_{0}(A)+\int_{\mathbb{R}}e^{i\theta x}-1-i\theta\tau(x)F_{A}(dx)\right),
\end{equation}
for every $\theta\in\mathbb{R}$ and $A\in\textbf{S}$, where $\nu_{0}(A)=\gamma^{(2)}(A)-\gamma^{(1)}(A)$ and $F_{A}(\cdot)=F_{A}^{(1)}(\cdot)+F_{A}^{(2)}(\cdot)$.
\begin{pro}\label{pro1}
	Let $\xi\in\mathcal{A}$ and adopt the notation above. Then, $F$ extends uniquely to a finite signed measure on $\textbf{S}\otimes\mathcal{B}(\mathbb{R})$.
\end{pro}
\begin{proof}
	Consider the notations above. For the first statement we need to show that $F$ is a finite signed measure on $\textbf{S}\otimes\mathcal{B}(\mathbb{R})$. Since $F^{(1)}$ is a finite measure on $\textbf{S}\otimes\mathcal{B}((0,\infty))$, it remains to show that $F^{(2)}$ is a finite signed measure on $\textbf{S}\otimes\mathcal{B}(\mathbb{R})$. We know that $F_{A}^{(2)}(\cdot)=\sum_{j=1}^{n}\delta_{s_{j}}(A)b_{j}(\cdot)$ where $b_{j}(\cdot)$ are finite signed measures on $\mathcal{B}(\mathbb{R})$. It is possible to see that $F^{(2)}$ is a bimeasure on $\textbf{S}\times\mathcal{B}(\mathbb{R})$ and that
	\begin{equation*}
	\sup\limits_{I}\sum_{i\in I}|F^{(2)}_{A_{i}}(B_{i})|=\sum_{j=1}^{n}|b_{j}|(\mathbb{R})<\infty,
	\end{equation*}
	where the supremum is taken over all the finite families of disjoints elements of $\textbf{S}\times\mathcal{B}(\mathbb{R})$. Then, by Theorem 5.18 in \cite{Pass} (see also Theorem 4 in \cite{Horo}) $F^{(2)}$ extends to a finite signed measure on $\textbf{S}\otimes\mathcal{B}(\mathbb{R})$. Thus, $F$ is a finite signed measure on $\textbf{S}\otimes\mathcal{B}(\mathbb{R})$.
\end{proof}
Following the notation of the ID case (see \cite{Kallenberg2} page 89), we call $F$ the \textit{quasi-L\'{e}vy measure} of $\xi$. Observe that in the ID case the L\'{e}vy measure might not even be $\sigma$-finite (see \cite{King} pages 82-83), while here our quasi-L\'{e}vy measure is a finite signed measure. Further, we remark that a similar result to Proposition \ref{pro1} holds for $\xi\in\mathcal{A}_{\infty}$ (see Remark \ref{rem-alpha-infinity}). In this case, $F^{(1)}$ is a measure (not necessarily $\sigma$-finite) on $\textbf{S}\otimes\mathcal{B}(\mathbb{R})$ (see Corollary 3.21 in \cite{Kallenberg2}), and $F^{(2)}$ is the same as in the proof of Proposition \ref{pro1}. Thus, in this case $F=F^{(1)}+F^{(2)}$ is a signed measure not necessarily finite.

In the following result we show the existence of a unique correspondence between any element in $\mathcal{A}$ and a characteristic pair.
\begin{thm}\label{th1}
	Let $\xi\in\mathcal{A}$. Then, there exists a pair $(\nu_{0}, F)$ s.t.~(\ref{2}) holds, where $\nu_{0}$ and $F$ are a finite signed measure on $\textbf{S}$ and $\textbf{S}\otimes\mathcal{B}(\mathbb{R})$, respectively, s.t.~for every $A\in\textbf{S}$ and $B\in\mathcal{B}(\mathbb{R})$:
	\\ \textnormal{(i)} $\nu_{0}(A)=-\gamma(A)+\sum_{j=1}^{n}\delta_{s_{j}}(A)c_{j}$, for some diffuse finite measure $\gamma$ on $\textbf{S}$, $c_{1},...,c_{n}\in\mathbb{R}$, and finitely many atoms $s_{1},...,s_{n}\in S$,
	\\ \textnormal{(ii)} $F(A\times B)=\tilde{G}(A\times B)+\sum_{j=1}^{n}\delta_{s_{j}}(A)b_{j}(B)$, for some finite measure $\tilde{G}$ on $\textbf{S}\otimes\mathcal{B}(\mathbb{R})$, which is the extension by zero of some measure $G$ on $\textbf{S}\otimes\mathcal{B}((0,\infty))$ with diffuse projections onto $S$, and for some finite signed measures $b_{j}$'s on $\mathcal{B}(\mathbb{R})$, such that $\exp(b_{1}),...,\exp(b_{n})$ are measures.
	
	Conversely, for every such pair $(\nu_{0},F)$ there exists a unique random measure $\xi\in\mathcal{A}$ s.t.~(\ref{2}) holds.
\end{thm}
\begin{proof}
	Concerning the atomless component of $\xi$, from Corollary 12.11 in \cite{Kallenberg} and Theorem 3.20 in \cite{Kallenberg2} we know that there exists a one to one correspondence between an ID atomless random measure with independent increments and a characteristic pair, composed by a diffuse measure on $\textbf{S}$ and a measure on $\textbf{S}\otimes\mathcal{B}(\mathbb{R})$ with diffuse projections onto $S$. In our case we note that the components of the characteristic pair are finite measures by definition.
	
	For the fixed component of $\xi$, by Theorem \ref{Cupp} we know that a characteristic triplet where the Gaussian component is zero and the quasi-L\'{e}vy measure is finite is the characteristic triplet of a QID random variable if and only if the exponential of the finite quasi-L\'{e}vy measure is a measure.
	
	Then, by the definition of $\xi$ and by the discussion and the computations at the beginning of this section on the characteristic functions of the components of $\xi$, we immediately obtain the result. 
	
	Notice that for the converse direction we need also to show the independence of the fixed and atomless components, but this follows immediately from the linear structure of $\nu_{0}$ and $F$.
\end{proof}
\begin{rem}\label{remark1}
	\textnormal{Notation: instead of using the characteristic pair we could have equivalently used the characteristic set $(\{s_{j}\}^{n}_{j=1},\gamma,\{c_{j}\}^{n}_{j=1},G,\{b_{j}\}^{n}_{j=1})$, with the above structure, in order to have a one to one identification with $\xi\in\mathcal{A}$.}
\end{rem}
\subsection{Properties of the dense class $\mathcal{A}'$}\label{Subsec-Properties-9}
Since $\mathcal{A}'\subsetneq\mathcal{A}$ all the results presented in the previous section holds for $\xi\in \mathcal{A}'$. In this subsection, we show that even better results holds for the elements in $\mathcal{A}'$. This is mainly due to the fact that we have more information about the structure of these random measures.

Let us recall Theorem 3.9 in \cite{LPS}. Despite we have used this theorem before we present it here to facilitate the reader in the understanding of the results of this subsection.
\begin{thm}
	[Theorem 3.9 in \cite{LPS}]\label{3.9} Let $\mu$ be a discrete distribution concentrated on $\{0, 1, 2,..., n\}$ for
	some $n \in \mathbb{N}$, \textit{i.e.}, $\mu=\sum_{j=0}^{n}a_{j}\delta_{j}$, where $a_{0},..., a_{n-1}\geq 0,$ $a_{n}> 0$, and $a_{0}+\cdots +	a_{n}= 1$. Then the following are equivalent:
	\\\textnormal{(i)} $\mu$ is quasi-infinitely divisible.
	\\\textnormal{(ii)} The characteristic function	of $\mu$ has no zeroes.
	\\\textnormal{(iii)} The polynomial $w\mapsto\sum_{j=0}^{n}a_{j}w^{j}$ in the complex variable $w$ has no roots on the unit circle, i.e.~$\sum_{j=0}^{n}a_{j}w^{j}\neq0$, for all $w\in\mathbb{C}$ with $|w| = 1$.
	
	Further, if one of the equivalent conditions (i)-(iii) holds, then the quasi-L\'{e}vy measure of $\mu$ is finite and concentrated on $\mathbb{Z}$, the drift lies in $\{0, 1 . . . , n\}$, and the
	Gaussian variance of $\mu$ is 0. More precisely, if $\xi_{1},...,\xi_{n}$ denote the $n$ complex roots of $w\mapsto\sum_{j=0}^{n}a_{j}w^{j}$, counted with multiplicity, then the quasi-L\'{e}vy measure of $\mu$ is given by
	\begin{equation}
	\nu=-\sum_{m=1}^{\infty}m^{-1}\bigg(\sum_{j:|\xi_{j}|<1}\xi_{j}^{m} \bigg)\delta_{-m}-\sum_{m=1}^{\infty}m^{-1}\bigg(\sum_{j:|\xi_{j}|>1}\xi_{j}^{-m} \bigg)\delta_{m},
	\end{equation}
	and the drift is equal to the number of those zeroes of this polynomial which lie inside the unit circle (counted with multiplicity), \textit{i.e.}, have modulus less than 1.
\end{thm}
In the following theorem we adopt the following notation. Let $\xi\in\mathcal{A}'$. We denote by $\alpha$ its atomless component and by $\beta_{j}$, $j=1,...,n$ the QID random variables of its fixed component, \textit{i.e.}~$\xi\stackrel{a.s.}{=}\alpha+\sum_{j=1}^{n}\beta_{j}\delta_{s_{j}}$. Further, for every $j=1,...,n$, we denote the law of $\beta_{j}$ by $\sum_{l=0}^{k_{j}}a_{j,l}\delta_{l}$, namely $\mathcal{L}(\beta_{j})=\sum_{l=0}^{k_{j}}a_{j,l}\delta_{l}$ and denote by $\zeta_{j,1},...,\zeta_{j,k_{j}}$ the $k_{j}$ complex roots of $w\mapsto\sum_{l=0}^{k_{j}}a_{j,l}w^{l}$. Finally, we denote by $b_{j}$ the quasi-L\'{e}vy measure of $\beta_{j}$, \textit{i.e.}
\begin{equation}\label{b9}
b_{j}=-\sum_{m=1}^{\infty}m^{-1}\bigg(\sum_{l:|\zeta_{j,l}|<1}\zeta_{j,l}^{m} \bigg)\delta_{-m}-\sum_{m=1}^{\infty}m^{-1}\bigg(\sum_{l:|\zeta_{j,l}|>1}\zeta_{j,l}^{-m} \bigg)\delta_{m},
\end{equation}
and by $c_{j}$ its drift, \textit{i.e.}~$c_{j}=\#\{|\zeta_{j,l}|<1,l=1,...,k_{j}\} $.
\begin{thm}\label{th1-9}
	Let $\xi\in\mathcal{A}'$. Then, there exists a pair $(\nu_{0}, F)$ s.t.~(\ref{2}) holds, where $\nu_{0}$ and $F$ are a finite signed measure on $\textbf{S}$ and $\textbf{S}\otimes\mathcal{B}(\mathbb{R})$, respectively, s.t.~for every $A\in\textbf{S}$ and $B\in\mathcal{B}(\mathbb{R})$:
	\\ \textnormal{(i)} $\nu_{0}(A)=\sum_{j=1}^{n}\delta_{s_{j}}(A)c_{j}$, where $n\in\mathbb{N}$, $s_{j}\in S$ is an atom, and $c_{j}=\#\{|\zeta_{j,l}|<1,l=1,...,k_{j}\}$, for $j=1,...,n$,
	\\ \textnormal{(ii)} $F(A\times B)=\tilde{G}(A\times B)+\sum_{j=1}^{n}\delta_{s_{j}}(A)b_{j}(B)$, where $\tilde{G}$ is a finite measure on $\textbf{S}\otimes\mathcal{B}(\mathbb{R})$ restricted on $S\times\mathbb{N}$ and with diffuse projections onto $S$, and where $b_{j}$ satisfies (\ref{b9}), for $j=1,...,n$.
	
	Conversely, for every such pair $(\nu_{0},F)$, where $\zeta_{j,1},...,\zeta_{j,k_{j}}$ denote the $k_{j}$ complex roots of some polynomial $w\mapsto\sum_{l=0}^{k_{j}}a_{j,l}w^{l}$ for $j=1,...,n$, there exists a unique random measure $\xi\in\mathcal{A}$ s.t.~(\ref{2}) holds.
\end{thm}
\begin{proof}
	It follows from the same arguments as the one used in Theorem \ref{th1} and from Theorem \ref{3.9}. In particular, the first direction is trivial. For the other direction, we have the following. As mentioned in the proof of Theorem \ref{th1}, we have a one-to-one correspondence for the atomless part of $\xi$ and its characteristic pair. Concerning the fixed component, let us assume that there exist $c_{j}$ and $b_{j}$ which are functions of some complex roots of some complex polynomial $w\mapsto\sum_{l=0}^{k_{j}}a_{j,l}w^{l}$ with no roots in the unite circle, where $k_{j}\in\mathbb{N}$, $a_{0},..., a_{k_{j}-1}\geq 0,$ $a_{k_{j}}> 0$, and $a_{0}+\cdots +	a_{k_{j}}= 1$. Then, by Theorem \ref{3.9} there exists a QID probability distribution $\mathcal{L}(\beta_{j})=\sum_{l=0}^{k_{j}}a_{j,l}\delta_{l}$. Since this holds for every $j=1,...,n$ then from the set of atoms $s_{1},...,s_{n}\in S$ we obtain the fixed component $\sum_{j=1}^{n}\delta_{s_{j}}\beta_{j}$ of a random measure in $\mathcal{A}'$.
\end{proof}
	The same comment in Remark \ref{remark1} for Theorem \ref{th1} holds here for Theorem \ref{th1-9}. In addition, we refer to \cite{Neh} for further properties of certain subclasses of point processes with quasi-L\'{e}vy measures.
\section{A Nonparametric Bayesian example}\label{Sec-Bay}
In this section we show how the setting and the results presented in Sections \ref{Sec-Atomless} and \ref{Sec-Properties} apply to a particular class of nonparametric prior distributions. The framework is the one of the paper by Broderick, Wilson and Jordan \cite{Bro1}. This framework is also explored in subsequent papers, see \cite{Bro2} among others. In their work they analyse Bayesian nonparametric prior and likelihood based on CRMs. In particular, they let the prior to be modelled as: 
\begin{equation*}
\Theta:=\sum_{k=1}^{K}\theta_{k}\delta_{\psi_{k}}
\end{equation*}
where the cardinality $K$ may be either finite or infinity and where $(\theta_{k},\psi_{k})$ is a pair consisting of the frequency (or rate) of the $k$-th trait together with its trait $\psi_{k}$, which belongs to some space $\Psi$ of traits. Further, they let the data point for the $m$-th individual to be modelled as:
\begin{equation*}
X_{m}:=\sum_{k=1}^{K_{m}}x_{m,k}\delta_{\psi_{k}}
\end{equation*}
where $x_{m,k}$ represents the degree to which the $m$-th data point belongs to the trait $\psi_{k}$.

This setting can be applied to many real world applications. In particular, in topic modelling we have that $\psi_{k}$ represents a topic; that is, $\psi_{k}$ is a distribution over words in a vocabulary. Further, $\theta_{k}$ might represent the frequency with which the topic $\psi_{k}$ occurs in a corpus of documents. Finally, $x_{j,k}$ represents the number of words in topic $\psi_{j,k}$ that occur in the $j$th document. So the $j$th document has a total length of $\sum_{k=1}^{K}x_{j,k}$ words. In this case, the actual observation consists of the words in each $m$ documents, and the topics of the whole corpus of documents are latent.

From a mathematical (and formal) point of view $\Theta$ and $X_{m}$ are defined as CRMs. In particular, for the data $X_{m}$, we let $x_{m,k}$ be drawn according to some distribution $H$ that takes $\theta_{k}$ as a parameter and have support on $\mathbb{Z}_{+}$, that is $x_{m,k}\stackrel{indep}{\sim}h(x_{m,k}|\theta_{k})$, independently across $m$ and $k$. We assume that $X_{1},...,X_{m}$ are i.i.d. conditional on $\Theta$. Moreover, \cite{Bro1} consider the following assumptions for $\Theta$ and $X_{m}$:
\\\\ \textit{Assumption A00}: the atomless component of $\Theta$ has characteristic pair $(\gamma,F)$ s.t.~$\gamma=0$ and $F(d\theta\times d\psi)=\nu(d\theta)\cdot G(d\psi)$, where $\nu$ is any $\sigma$-finite measure on $\mathbb{R}_{+}$ and $G$ is a proper distribution on $\Psi$ with no atoms. 
\\\\ \textit{Assumptions A0, A1, and A2}: $\Theta$ has a finite number of fixed atoms, $\nu(\mathbb{R}_{+})=\infty$, and $\sum_{x=1}^{\infty}\int_{\mathbb{R}_{+}}h(x|\theta)\nu(d\theta)$ $<\infty$, respectively.
\\\\
We remark that by Assumption A00 we have that the location of the non-fixed atoms $\psi$ and the frequencies $\theta_{k}$ are stochastically independent. We call $\nu$ the \textit{weights rate measure} of $\Theta$. Moreover, the assumptions A0, A1 and A2 comes from a modelling need. By assuming A0 we are saying that we initially know certain traits, by A1 that there are a countable infinity of possible traits, and by A2 that the amount of information from finitely represented data is finite (because by A2 the number of non-fixed atoms is finite).

The first main result in \cite{Bro1} is Theorem 3.1, which shows explicit formulations for the posterior distribution $\Theta|X_{1}$, and it is extended in Corollary 3.2 to the posterior $\Theta|X_{1:m}$. In the following result we are going to show that similar results hold for any random measure in $\mathcal{A}$ without assuming A0, A1 or A2.

Notice that we can write $\Theta=\sum_{k=1}^{K}\theta_{k}\delta_{\psi_{k}}$, where $K=K_{fix}+K_{ord}$, namely $K$ is the sum of the fixed and non-fixed atoms, thus $K$ is random. Following the notation of \cite{Bro1}, we denote the fixed component of $\Theta$ by $\Theta_{fix}=\sum_{k=1}^{K_{fix}}\theta_{fix,k}\delta_{\psi_{fix,k}}$ and the law of $\theta_{fix,k}$ by $F_{fix,k}:=\mathcal{L}(\Theta(\{\psi_{fix,k}\}))$.
\begin{pro}\label{pro-Bay}
	Let $\Theta\in\mathcal{A}$ satisfying $A00$. Write $\Theta=\sum_{k=1}^{K}\theta_{k}\delta_{\psi_{k}}$, and let $X_{1}, . . . , X_{m}$ be generated conditional on $\Theta$ according to $X_{1}:=\sum_{k=1}^{K}x_{1,k}\delta_{\psi_{k}}$ with $x_{1,k}\stackrel{indep}{\sim} h(x|\theta_{k})$ for proper,	discrete probability mass function $h$. It is enough to make the assumption for $X_{1}$ since the $X_{1},...,X_{m}$ are i.i.d. conditional on $\Theta$.
	
	Then let $\Theta_{post}$ be a random measure with the distribution of $\Theta|X_{1:m}$ (\textit{i.e.}~$\Theta|X_{1},...,X_{m}$). $\Theta_{post}$ is a CRM with three parts.
	
	\textnormal{1.} For each $k\in [K_{fix}]$, $\Theta_{post}$ has a fixed atom at $\psi_{fix,k}$ with weight $\theta_{post,fix,k}$  distributed according to the finite-dimensional posterior $F_{post,fix,k} (d\theta)$ that comes from prior
	$F_{fix,k}$ , likelihood $h$, and observation $X({\psi_{fix,k}})$. Moreover, $F_{fix,k}$ is QID with no Gaussian component and finite quasi-L\'{e}vy measure, and $F_{post,fix,k}(d\theta)\varpropto F_{fix,k}(d\theta)\prod_{j=1}^{m}h(x_{fix,j,k}|\theta)$.
	
	\textnormal{2.} Let $\{\psi_{new,k}: k \in [K_{new} ]\}$ be the union of atom locations across $X_{1}, X_{2}, . . . , X_{m}$ minus the
	fixed locations in the prior of $\Theta$. $K_{new}$ is finite. Let $x_{new,j,k}$ be the weight of the atom in $X_{j}$ located at $\psi_{new,k}$, for some $j=1,...,m$. Then $\Theta_{post}$ has a fixed atom at $x_{new,k}$ with	random weight $\theta_{post,new,k}$, whose distribution $F_{post,new,k}(d\theta)\varpropto\nu(d\theta)\prod_{j=1}^{m}h(x_{new,j,k}|\theta)$.
	
	\textnormal{3.} The ordinary component of $\Theta_{post}$ has finite weights rate measure $\nu_{post,m}(d\theta):=\nu(d\theta)h(0|\theta)^{m}$.
\end{pro}
\begin{rem}
	\textnormal{Observe that since $\Theta\in\mathcal{A}$ then it has finite fixed atoms so assumption A0 is satisfied. Moreover, since $\nu$ is also finite and $h(x|\theta)\leq 1$, then assumption A2 is also satisfied. The only difference with Theorem 3.1 and Corollary 3.2 in \cite{Bro1} is that we do not necessarily satisfy assumption A1. However, A1 is a modelling assumption rather than a technical one. Indeed, the proof of this result follows from similar arguments as the one used in the proof of Theorem 3.1 and Corollary 3.2 in \cite{Bro1}. We write them for completeness.}
\end{rem}
\begin{proof}
	Let us first prove the result for $\Theta|X$. Any fixed atom $\theta_{fix,k}\delta_{\psi_{fix,k}}$ in the prior is independent of the other fixed atoms and of the ordinary component. Thus, all of $X$ except $x_{fix,k}:=X(\{\psi_{fix,k}\})$ is independent of $\theta_{fix,k}$. Thus, $\Theta|X$ has a fixed atom at $\psi_{fix,k}$ and $\mathcal{L}(\theta_{post,fix,k})\varpropto F_{fix,k}(d\theta)h(x_{fix,k}|\theta)$. Recall that since $G$ is continuous, all the fixed and non-fixed atoms of $\Theta$ are at a.s.~distinct locations. Observe that by letting $\Psi_{fix}:=\{\psi_{fix,1},...,\psi_{fix,K_{fix}}\}$ we can define the fixed and ordinary component of $X$ by $X_{fix}(A):=X(A\cap \Psi_{fix})$ and $X_{ord}(A):=X(A\cap (\Psi\setminus\Psi_{fix}))$, respectively.
	
	Let $x\in\mathbb{Z}_{+}$ and let $\{\psi_{new,x,1},...,\psi_{new,x,K_{new,x}}\}$ be all the locations of atoms in $X_{ord}$ of size $x$, which is finite and it is a subset of the locations of atoms of $\Theta_{ord}$. Further, let $\theta_{new,x,k}:=\Theta(\{\psi_{new,x,k}\})$. Observe that the values $\{\theta_{new,x,k} \}_{k=1}^{K_{new,x}}$ are generated from a thinned Poisson point process with rate measure (also known as intensity measure) $\nu_{x}(d\theta)=\nu(d\theta)h(x|\theta)$, this is due to the $h(x|\theta)$-thinning of the Poisson point process $\{\theta_{ord,k}\}_{k=1}^{K_{ord}}$ which has rate measure $\nu$. Moreover, given that $\nu_{x}(\mathbb{R}_{+})<\infty$, we have that $\mathcal{L}(\theta_{new,x,k})\varpropto\nu(d\theta)h(x|\theta)$. Finally, observe that there is a possibility that atoms in $\Theta_{ord}$ are not observed in $X_{ord}$, this happens when the likelihood draw returns a zero. These atom weights form a Poisson point process with rate measure $\nu(d\theta)h(0|\theta)$.
	
	Considering $\Theta|X_{1}$ as the new prior we obtain the formulation for the posterior $\Theta|X_{1},X_{2}$ by induction and by observing that the assumptions are still satisfied by $\Theta|X_{1}$. Then, by induction we conclude the proof.
\end{proof}
In the next result, we show that random measures in $\mathcal{A}$ satisfying A00 are dense in the space of all CRMs satisfying A0, A1 and A2, namely all the random measures considered in \cite{Bro1} (and in \cite{Bro2}). Further, we show how this result translates into a convergence for the ordinary component of the posterior of these random measures.
\begin{pro}\label{pro-Bay-2}
	Consider any random measure $\Theta$ satisfying A00, A0, A1 and A2, namely as in Theorem 3.1 in \cite{Bro1}. Then, there exists a sequence of random measures $(\Theta_{n})_{n\in\mathbb{N}}$ in $\mathcal{A}$ and satisfying A00 such that $\Theta_{n}\stackrel{d}{\to}\Theta$, as $n\to\infty$. Further, $\Theta_{n,post,ord}\stackrel{d}{\to}\Theta_{post,ord}$, as $n\to\infty$.
\end{pro}
\begin{proof}
	The first part of this proof consists in realising that the arguments in the proof of Proposition \ref{pro-Poisson} and Theorem \ref{density-T2} adapt to the present case.
	
	Denote by $F(d\theta\times d\psi)=\nu(d\theta)\cdot G(d\psi)$ the L\'{e}vy measure of $\Theta$. Following the proofs of Proposition \ref{pro-Poisson} and Theorem \ref{density-T2} it is possible to see that the approximating sequence $\Theta_{n}$ should have L\'{e}vy measure $\nu_{n}(d\theta)\cdot G_{n}(d\psi)$ where $\nu_{n}(d\theta):=\nu((\frac{1}{n},\infty)\cap d\theta)$ and $G_{n}(d\theta):=G(S_{n}\cap d\psi)$. However, given the assumptions on $F$, namely that $G$ is a finite measure, we can (and we do) take the L\'{e}vy measure of $\Theta_{n}$ to be given by $F_{n}(d\theta\times d\psi):=\nu_{n}(d\theta)\cdot G(d\psi)$. Then, applying the same arguments as the one used in the proof of Proposition \ref{pro-Poisson} and Theorem \ref{density-T2}, we obtain that the ordinary component of $\Theta_{n}$ converge in distribution to the one of $\Theta$. The convergence of the fixed component follows directly from Theorem \ref{density-T2}. Since $F_{n}$ is finite, we have that $\Theta_{n}$ is in $\mathcal{A}$ and that it satisfies A00. 
	
	For the convergence of the posteriors, consider $\Theta_{n}$ with its respective data points $X_{n,1},....,X_{n,m}$, which are defined conditional on $\Theta_{n}$ as in Proposition \ref{pro-Bay} and belong to some probability spaces possibly different from the one of the other data points. From Proposition \ref{pro-Bay} we have that $\Theta_{n,post}$ has finite weights rate measure $\nu_{n,post,m}(d\theta):=\nu_{n}(d\theta)h(0|\theta)^{m}$, while from Corollary 3.2 in \cite{Bro1} we know that $\Theta_{post}$ has finite weights rate measure $\nu_{post,m}(d\theta):=\nu(d\theta)h(0|\theta)^{m}$. Since $\nu_{n,post,m}(\cdot)=\nu_{post,m}((\frac{1}{n},\infty)\cap\cdot)$ we obtain the result by Proposition \ref{pro-Poisson}.
\end{proof}
	We summarise our findings so far in words. First, we obtain an explicit expression for the posterior of any random measure in $\mathcal{A}$ satisfying A00. Second, such random measures are dense with respect to convergence in distribution in the space of all priors considered in \cite{Bro1}. Third, when approximating in distribution such a prior, the ordinary component of the posteriors of these random measures converge to the one of the prior. 
	
	Thus, by these results we have a random truncation procedure; this is so because the number of non-fixed atoms of the prior is random and almost surely finite for every $n\in\mathbb{N}$. Thus, the present truncation procedure extends the one of \cite{Bro2}. Indeed, we do not arbitrarily fix the number non-fixed atoms of the truncated prior and we are able to keep explicit formulations for the posterior of the truncated prior.

In the next result we show that, under certain conditions, we have automatic conjugacy for random measures in $\mathcal{A}'$ satisfying A00.
\begin{pro}\label{pro-Bay-3}
Let $\Theta\in \mathcal{A}'$ satisfying A00 and with weights rate measure having finite support. Let $X$ be generated conditional on $\Theta$ according to $X:=\sum_{k=1}^{K}x_{k}\delta_{\psi_{k}}$ with $x_{k}\stackrel{indep}{\sim} h(x|\theta_{k})$ for proper, discrete probability mass function $h$. Assume that the characteristic functions of the random variables of the fixed component of $\Theta_{post}$ have no zeros, namely assume that for every $x\in\mathbb{N}$, $z\in\mathbb{R}$ and $k\in [K_{fix}]$
\begin{equation}\label{ass-Bay}
\int_{0}^{\infty}e^{iz\theta}h(x|\theta)F_{fix,k}(d\theta)\neq0\quad\text{and} \quad\int_{0}^{\infty}e^{iz\theta}h(x|\theta)\nu(d\theta)\neq0.
\end{equation}
Then, $\Theta_{post}\in \mathcal{A}'$, satisfies A00 and has weights rate measure with finite support.
\end{pro}
\begin{proof}
	Assumption (\ref{ass-Bay}) implies that the characteristic functions of $F_{post,fix,k}$ and of $F_{post,new,j}$ have no zeros. Further, they are also supported on a finite subset of $\mathbb{Z}_{+}$. Then, by Theorem \ref{3.9} we obtain the result.
\end{proof}
\begin{rem}\label{rem-Bay}
	\textnormal{Let $\Theta$ and $X$ be as in Proposition \ref{pro-Bay-3}. Notice that we can write $F_{fix,k}=\sum_{j=0}^{n^{(k)}}a^{(k)}_{j}\delta_{j}$, where $a_{0},..., a_{n-1}\geq 0,$ $a_{n}> 0$, and $a_{0}+\cdots +	a_{n}= 1$, for $k\in K_{fix}$. Further, we can write $\nu=\sum_{j=1}^{K_{\nu}}b_{j}\delta_{j}$, where $K_{\nu}\in\mathbb{N}$ indicates the highest value in $\textnormal{supp}(\nu)$, $b_{1},...,b_{K_{\nu}-1}\geq0$ and $b_{K_{\nu}}>0$. Assumption (\ref{ass-Bay}) can be rewritten as: For every $x\in\mathbb{N}$, $z\in\mathbb{R}$ and $k\in [K_{fix}]$, assume that
	\begin{equation*}
	\sum_{j=0}^{n^{(k)}}e^{izj}h(x|j)a_{j}^{(k)}\neq0\quad\text{and} \quad\sum_{j=1}^{K_{\nu}}e^{izj}h(x|j)b_{j}\neq0.
	\end{equation*}
	Moreover, by Theorem \ref{3.9} this assumption (and so assumption (\ref{ass-Bay})) is equivalent to the following assumption: For every $x\in\mathbb{N}$ and $k\in [K_{fix}]$, assume that the polynomials $w\mapsto\sum_{j=0}^{n^{(k)}}h(x|j)a_{j}^{(k)}w^{j}$ and $w\mapsto\sum_{j=1}^{K_{\nu}}h(x|j)b_{j}w^{j}$ in the complex variable $w$ have no roots on the unit circle.}
\end{rem}
\begin{rem}
\textnormal{The results presented in this section holds also if the weights rate measure is infinite, namely $\nu(\mathbb{R}_{+})=\infty$ (under the additional assumptions A1 and A2). In particular, the equivalent of Proposition \ref{pro-Bay} would be identical to Corollary 3.2 except for the result of point 1, because here we additionally know that $F_{fix,k}$ is QID with no Gaussian component and finite quasi-L\'{e}vy measure. Further, the equivalent of Proposition \ref{pro-Bay-2} would follows from the arguments presented taking into consideration Remark \ref{rem-alpha-infinity}. The equivalent of Proposition \ref{pro-Bay-3} is more subtle and it is presented below.}
\end{rem}
Consider the following class of QID random measures:
\begin{equation*}
\mathcal{A}'':=\bigg\{\xi\in\mathcal{I}\bigg|\xi\stackrel{a.s.}{=}\alpha+\sum_{j=1}^{K}\beta_{j}\delta_{s_{j}},\textnormal{with $\alpha$ an atomless point process with independent increments}  
\end{equation*}
\begin{equation*}
\textnormal{and finite L\'{e}vy measure, $\{s_{j}:j=1,...,K\}$ a finite set of fixed atoms in $S$, and $\beta_{j}$, $j\geq1$, $\mathbb{Z}_{+}$-valued}
\end{equation*}
\begin{equation*}
\textnormal{QID random variables that are mutually independent and independent of $\alpha$}\bigg\}.
\end{equation*}
Let $\mathcal{A}''_{\infty}$ indicate the set of random measures like in $\mathcal{A}$ but with $\alpha$ being any atomless point process with independent increments. As a side comment, we remark that is possible to see that a result similar to Theorem \ref{th1} and Theorem \ref{th1-9} holds for the elements in $\mathcal{A}''$, where thanks to Theorem 8.1 in \cite{LPS} we are able to know the structure of their L\'{e}vy-Khintchine representation in more details.
\begin{pro}
	Let $\Theta\in\mathcal{A}''_{\infty}$ and assume A00, A0, A1 and A2. Let $X$ be generated conditional on $\Theta$ according to $X:=\sum_{k=1}^{\infty}x_{k}\delta_{\psi_{k}}$ with $x_{k}\stackrel{indep}{\sim} h(x|\theta_{k})$ for proper, discrete probability mass function $h$. Assume that the characteristic functions of the random variables of the fixed component of $\Theta_{post}$ have no zeros, namely assume that for every $x\in\mathbb{N}$, $z\in\mathbb{R}$ and $k\in [K_{fix}]$
	\begin{equation}\label{ass-Bay-2}
	\int_{0}^{\infty}e^{iz\theta}h(x|\theta)F_{fix,k}(d\theta)\neq0\quad\text{and} \quad\int_{0}^{\infty}e^{iz\theta}h(x|\theta)\nu(d\theta)\neq0.
	\end{equation}
	Then, $\Theta_{post}\in \mathcal{A}''_{\infty}$ and satisfies A00, A0, A1 and A2.
\end{pro}
\begin{proof}
Assumption (\ref{ass-Bay-2}) implies that the characteristic functions of $F_{post,fix,k}$ and of $F_{post,new,j}$ have no zeros. Further, they are also supported on $\mathbb{Z}_{+}$. Then, by Theorem 8.1 in \cite{LPS} we obtain the result.
\end{proof}
Observe that assumption (\ref{ass-Bay-2}) can be rewritten more explicitly as done in Remark \ref{rem-Bay} for assumption (\ref{ass-Bay}).
\section*{Acknowledgement}
The author would like to thank Almut Veraart, Fabio Bernasconi and Ismael Castillo for useful discussions. The research developed in this paper is supported by the EPSRC (award ref.~1643696) at Imperial College London and by the Fondation Sciences Math\'{e}matiques de Paris (FSMP) fellowship, held at LPSM (Sorbonne University).

\end{document}